\documentclass[12pt, reqno]{jams-l}

\usepackage{amsmath,enumerate,amsfonts,amssymb,color,graphicx,amsthm}

\usepackage[colorlinks=true, allcolors=black]{hyperref}
\usepackage{todonotes}
\usepackage[normalem]{ulem}
\numberwithin{equation}{section}
\usepackage{cite}
\usepackage{mathtools}
\usepackage[toc,page]{appendix}
\usepackage[margin=3cm]{geometry}

\newlength{\leftstackrelawd}
\newlength{\leftstackrelbwd}
\def\leftstackrel#1#2{\settowidth{\leftstackrelawd}%
	{${{}^{#1}}$}\settowidth{\leftstackrelbwd}{$#2$}%
	\addtolength{\leftstackrelawd}{-\leftstackrelbwd}%
	\leavevmode\ifthenelse{\lengthtest{\leftstackrelawd>0pt}}%
	{\kern-.5\leftstackrelawd}{}\mathrel{\mathop{#2}\limits^{#1}}}

\theoremstyle{plain}
\newtheorem{thm}{Theorem}[section]
\newtheorem{lem}[thm]{Lemma}
\newtheorem{cor}[thm]{Corollary}
\newtheorem{prop}[thm]{Proposition}
\newtheorem*{thm*}{Theorem}

\theoremstyle{definition}

\newtheorem{rmk}[thm]{Remark}
\newtheorem{?}[thm]{Problem}

\renewcommand{\phi}{\varphi}
\renewcommand{\epsilon}{\varepsilon}
\makeatletter
\def\@cite#1#2{[\textbf{#1\if@tempswa , #2\fi}]}
\def\@biblabel#1{[\textbf{#1}]}
\makeatother

\makeatletter
\newcommand*{\defeq}{\mathrel{\rlap{%
			\raisebox{0.3ex}{$\m@th\cdot$}}%
		\raisebox{-0.3ex}{$\m@th\cdot$}}%
	=}
\makeatother

\makeatletter
\newcommand*{\eqdef}{=\mathrel{\rlap{%
			\raisebox{0.3ex}{$\m@th\cdot$}}%
		\raisebox{-0.3ex}{$\m@th\cdot$}}%
}
\makeatother

\newcounter{marnote}

\makeatletter
\def\underbracex#1#2{\mathop{\vtop{\m@th\ialign{##\crcr
				$\hfil\displaystyle{#2}\hfil$\crcr
				\noalign{\kern3\p@\nointerlineskip}%
				#1\crcr\noalign{\kern3\p@}}}}\limits}

\def\upbracefilla{$\m@th \setbox\z@\hbox{$\braceld$}%
	\bracelu\leaders\vrule \@height\ht\z@ \@depth\z@\hfill 
	\kern\p@\vrule \@width\p@\kern\p@\vrule \@width\p@\kern\p@\vrule \@width\p@
	$}

\def\upbracefillb{$\m@th \setbox\z@\hbox{$\braceld$}%
	\vrule \@width\p@\kern\p@\vrule \@width\p@\kern\p@\vrule \@width\p@\kern\p@
	\leaders\vrule \@height\ht\z@ \@depth\z@\hfill\bracerd
	\braceld\leaders\vrule \@height\ht\z@ \@depth\z@\hfill
	\kern\p@\vrule \@width\p@\kern\p@\vrule \@width\p@\kern\p@\vrule \@width\p@
	$}

\def\upbracefillc{$\m@th \setbox\z@\hbox{$\braceld$}%
	\vrule \@width\p@\kern\p@\vrule \@width\p@\kern\p@\vrule \@width\p@\kern\p@
	\leaders\vrule \@height\ht\z@ \@depth\z@\hfill
	\kern\p@\vrule \@width\p@\kern\p@\vrule \@width\p@\kern\p@\vrule \@width\p@
	$}

\def\upbracefilld{$\m@th \setbox\z@\hbox{$\braceld$}%
	\vrule \@width\p@\kern\p@\vrule \@width\p@\kern\p@\vrule \@width\p@\kern\p@
	\leaders\vrule \@height\ht\z@ \@depth\z@\hfill\braceru$}

\def\upbracefillbd{$\m@th \setbox\z@\hbox{$\braceld$}%
	\vrule \@width\p@\kern\p@\vrule \@width\p@\kern\p@\vrule \@width\p@\kern\p@
	\bracerd\braceld
	\leaders\vrule \@height\ht\z@ \@depth\z@\hfill\braceru$}

\makeatother

\makeatletter
\def\l@subsection{\@tocline{2}{0pt}{2.5pc}{5pc}{}}
\makeatother

\begin{document}

	\title[Interior second derivative estimates for augmented Hessian equations]{Augmented Hessian equations on Riemannian manifolds: from integral to pointwise local \linebreak second derivative estimates} 
	
	\author{Jonah A. J. Duncan}
	\address{Johns Hopkins University, 404 Krieger Hall, Department of Mathematics, 3400 N. Charles Street, Baltimore, MD 21218, US.}
	\curraddr{}
	\email{jdunca33@jhu.edu}
	\thanks{}

	\date{}

	\vspace*{-1mm}\maketitle

	\vspace*{-6mm}\begin{abstract}
	We obtain \textit{a priori} local pointwise second derivative estimates for \linebreak solutions $u$ to a class of augmented Hessian equations on Riemannian manifolds, in terms of the $C^1$ norm and certain $W^{2,p}$ norms of $u$. We consider the case that no structural assumptions are imposed on either the augmenting term or the right hand side of the equation, and the case where these terms are convex in the gradient \linebreak variable. In the latter case, under an additional ellipticity condition we prove that the dependence on any $W^{2,p}$ norm can be dropped. Our results are derived using integral estimates.

		
	\end{abstract}

\section{Introduction}

In this paper, we obtain local pointwise second derivative estimates on solutions $u$ to a general class of fully nonlinear, non-uniformly elliptic equations in terms of the $C^1$ norm and certain $W^{2,p}$ norms of $u$. More precisely, for a domain $\Omega$ contained inside a smooth Riemannian manifold $(M^n,g_0)$ of dimension $n\geq 2$, we consider equations of the form
\begin{equation}\label{1}
f\big(\lambda\big(\nabla_{g_0}^2 u - A[u]\big)\big) = B[u], \quad \lambda\big(\nabla_{g_0}^2 u - A[u]\big)\in\Gamma \quad \text{ on }\Omega,
\end{equation}
where the pair $(f,\Gamma)$ is assumed to satisfy the following standard properties: 
\begin{align}
& \Gamma\subset\mathbb{R}^n\text{ is an open, convex, connected symmetric cone with vertex at 0}, \label{21} \\
& \Gamma_n^+ = \{\lambda\in\mathbb{R}^n: \lambda_i > 0 ~\forall ~1\leq i \leq n\} \subseteq \Gamma \subseteq \Gamma_1^+ =  \{\lambda\in\mathbb{R}^n : \lambda_1+\dots+\lambda_n > 0\}, \label{22} \\
& f\in C^\infty(\Gamma)\cap C^0(\overline{\Gamma}) \text{ is concave, 1-homogeneous and symmetric in the }\lambda_i, \label{23}  \\
& f>0 \text{ in }\Gamma, \quad f = 0 \text{ on }\partial\Gamma, \quad f_{\lambda_i} >0 \text{ in } \Gamma \text{ for }1 \leq i \leq n. \label{24}
\end{align}
In \eqref{1}, $\nabla_{g_0}^2 u$ denotes the Hessian of $u$ with respect to $g_0$, 
\begin{equation*}
A[u](x) = A(x,u(x),du(x))\in\operatorname{Sym}(T^*_x\Omega\otimes T^*_x\Omega)
\end{equation*}
defines a symmetric $(0,2)$-tensor at each point $x\in \Omega$, $B[u](x) = B(x,u(x),du(x))$ takes real positive values, and for a $(1,1)$-tensor $T$ on $\Omega$, we denote by $\lambda(T(x))\in\mathbb{R}^n$ the eigenvalues of $T$ at $x$. Note that we are implicitly raising an index of the $(0,2)$-tensor $\nabla_{g_0}^2 u - A[u]$ in \eqref{1} using the inverse metric $g_0^{-1}$, and we continue to follow this convention throughout the paper. 

Typical examples of $(f,\Gamma)$ satisfying \eqref{21}--\eqref{24} are given by $(\sigma_k^{1/k},\Gamma_k^+)$ for $1\leq k \leq n$, where $\sigma_k$ is the $k$'th elementary symmetric polynomial, defined by 
\begin{equation*}
\sigma_k(\lambda_1,\dots,\lambda_n) = \sum_{1\leq i_1<\dots<i_k\leq n}\lambda_{i_1}\dots\lambda_{i_k},
\end{equation*}
and $\Gamma_k^+ = \{\lambda\in\mathbb{R}^n:\sigma_j(\lambda)>0\text{ for all }1\leq j \leq k\}$. In particular, $\sigma_1(\lambda(\cdot))$ is the trace operator and $\sigma_n(\lambda(\cdot))$ is the determinant operator.

Since the work of Caffarelli, Nirenberg \& Spruck \cite{CNS3} on fully nonlinear Hessian equations (corresponding to $A\equiv 0$ in \eqref{1}), there has been a significant amount of study devoted to equations of the form \eqref{1}. Such equations arise in the theory of optimal transport \cite{MTW},  geometric optics \cite{Wan96} and conformal geometry \cite{Via00a}, for example. We refer to \eqref{1} as an \textit{augmented Hessian equation}, following the recent study of boundary value problems associated to \eqref{1} on Euclidean domains by Jiang, Trudinger \& Yang \cite{JTY15} and Jiang \& Trudinger \cite{JT17, JT18, JT19, JT20}.

Crucial in the study of \eqref{1} are \textit{a priori} estimates on solutions. As is well-known, structural properties of the term $A$ play an important role in obtaining \textit{a priori} estimates for \eqref{1}, and hence also in the existence and regularity theory for \eqref{1}. Indeed, an example of Heinz-Lewy (see e.g.~Section 9.5 in \cite{Sch90}) demonstrates that \eqref{1} does not admit local pointwise second derivative estimates for certain choices of $A$. We are interested in the related problem of obtaining such estimates on solutions to \eqref{1} in terms of their $C^1$ and $W^{2,p}$ norms, in the absence of any structural assumptions on $A$ or $B$ and under only mild assumptions on $f$ and $\Gamma$. Our main motivation for considering this problem is twofold: \medskip 
\begin{enumerate}
	\item The problem of obtaining local pointwise second derivative estimates for \eqref{1} is then reduced to obtaining local $C^1$ and $W^{2,p}$ estimates for suitable $p$ (this may be useful for the existence theory -- see e.g.~Remark \ref{89}). \medskip 
	\item Such a result demonstrates that once $C^1$ estimates are obtained, any \textit{failure} of local pointwise second derivative estimates for \eqref{1} must actually occur at the $W^{2,p}$ level for suitable $p$.\medskip 
\end{enumerate}
Our lower bounds imposed on $p$ will be explicit in many cases, and we will also show that these lower bounds can be improved in special cases, e.g.~when $A$ and $B$ are convex in the gradient variable.

In addition to \eqref{23} and \eqref{24}, we introduce one more condition on $(f,\Gamma)$, which is related to the lower bound on the Sobolev exponent $p$ imposed in our first main result. To formulate this condition, we fix $(f,\Gamma)$ satisfying \eqref{21}--\eqref{24}, and for a symmetric matrix $A$, we denote by $F(A)$ the matrix with entries
\begin{equation*}
F(A)_i^j = \frac{\partial }{\partial A_j^i}f(\lambda(A)).
\end{equation*}
By \eqref{24}, $F(A)$ is positive definite if $\lambda(A)\in\Gamma$. Our condition is then as follows: there exist constants $C>0$ and $\gamma\geq 0$ (depending only on $(f,\Gamma)$) such that
\begin{equation}\label{26}
\frac{[\operatorname{tr}(F(A))]^n}{\det (F(A))} \leq C \bigg(\frac{\operatorname{tr}(A)}{f(\lambda(A))}\bigg)^\gamma \quad \text{for all }A \text{ with }\lambda(A)\in\Gamma.
\end{equation}
This condition was previously used by the author in \cite{D22}, and is satisfied in many important cases: \medskip 

\noindent\textbf{Example 1:} When $(f,\Gamma) = ((\sigma_k/\sigma_l)^{1/(k-l)},\Gamma_k^+)$ for some $0 \leq l < k \leq n$ and $2\leq k \leq n$ (with the convention that $\sigma_0=1$),  \eqref{26} is satisfied with $\gamma = (k-1)\max\{k-l, 2\}$ -- see \cite[Proposition 4.2]{BL}. In particular, when $(f,\Gamma) = (\sigma_k^{1/k},\Gamma_k^+)$ for $2\leq k \leq n$, \eqref{26} is satisfied with $\gamma = k(k-1)$. \medskip 

\noindent\textbf{Example 2:} When, in addition to \eqref{21} and \eqref{22}, the cone $\Gamma$ satisfies $(1,0,\dots,0)\in\Gamma$, \eqref{26} is satisfied with $\gamma =0$. This follows from the fact that $\frac{\partial f}{\partial \lambda_i}$ are the eigenvalues of $F$ and \cite[Proposition A.1]{LN20}, which asserts in this case the existence of a constant $\nu\in(0,1)$ such that 
\begin{equation}\label{85}
\frac{\partial f}{\partial\lambda_i}(\lambda) \geq \nu \sum_{j=1}^n\frac{\partial f}{\partial \lambda_j}(\lambda)\quad\text{for all }i=1,\dots,n \text{ and }\lambda\in\Gamma.
\end{equation}

\noindent We note that $\Gamma$ satisfies \eqref{21}, \eqref{22} and $(1,0,\dots,0)\in\Gamma$ if and only if there exists a cone $\widetilde{\Gamma}$ satisfying \eqref{21} and \eqref{22} and a number $\tau\in[0,1)$ such that $\Gamma = (\widetilde{\Gamma})^\tau$, where
\begin{align}
(\widetilde{\Gamma})^\tau \defeq \{\lambda\in\mathbb{R}^n:\tau\lambda + (1-\tau)\sigma_1(\lambda)(1,\dots,1)\in\widetilde{\Gamma}\}. \nonumber
\end{align}
We refer the reader to the appendix in \cite{DN22} for a proof of this statement. Moreover, if the pair $(f,\Gamma)$ satisfies \eqref{21}--\eqref{24}, then so does the pair $(f^\tau,\Gamma^\tau)$, where $f^\tau(\lambda)\defeq f(\tau\lambda + (1-\tau)\sigma_1(\lambda)(1,\dots,1))$. Clearly, $\Gamma^1 = \Gamma$ and $\Gamma^0 = \Gamma_1^+$.

Our first main result is as follows:

\begin{thm}\label{A}
	Let $\Omega$ be a bounded domain contained in a smooth Riemannian manifold $(M^n,g_0)$ of dimension $n\geq 2$. Suppose that $(f,\Gamma)$ satisfies \eqref{21}--\eqref{24}, $\Gamma\subseteq(\Gamma_2^+)^\tau$ for some $\tau\in(0,1]$, \eqref{26} holds for some constants $C>0$ and $\gamma\geq 0$, $A = A(x,z,\xi)\in C_{\operatorname{loc}}^{2}(\Omega\times\mathbb{R}\times T^*\Omega)$ is $\operatorname{Sym}(T^*\Omega\otimes T^*\Omega)$-valued and $B =B(x,z,\xi) \in C_{\operatorname{loc}}^{2}(\Omega\times\mathbb{R}\times T^*\Omega)$ is real-valued and positive. Finally, suppose that one of the following statements holds:
	\begin{enumerate}
		\item $A$ and $B$ are convex in $\xi$ and $p>\gamma$.\medskip 
		\item $p>n+\gamma$.  \medskip 
	\end{enumerate}
	Then if $u\in C^4(\Omega)$ is a solution to \eqref{1} and $\Omega'\Subset\Omega$, it holds that
	\begin{equation}\label{53}
	\|\nabla_{g_0}^2u\|_{L^\infty(\Omega')} \leq C
	\end{equation}
	where $C$ is a constant depending only on $f, \Gamma, n, p, g_0, A, B, \Omega, \Omega'$ and an upper bound for $\|\nabla_{g_0}^2 u\|_{L^p(\Omega)} + \|u\|_{C^1(\Omega)}$. Moreover, if $(f,\Gamma) = ((\sigma_2^{1/2})^\tau,(\Gamma_2^+)^\tau)$ for some $\tau\in(0,1]$, the convexity assumption on $B$ in Case 1 can be dropped. 
\end{thm}

\begin{rmk}\label{63}
	The assumption $\Gamma\subseteq (\Gamma_2^+)^\tau$ for some $\tau\in(0,1]$ in Theorem \ref{A} implies a pointwise bound for $|\nabla_{g_0}^2 u|$ in terms of a pointwise bound for $|\Delta_{g_0} u|$ and an upper bound for $\|u\|_{C^1(\Omega)}$; this is clear from the formula $\sigma_2(A) = \frac{1}{2}(\operatorname{tr}(A)^2 - |A|^2)$. Moreover, the weaker assumption $\Gamma\subseteq \Gamma_1^+$ immediately implies a lower bound for $\Delta_{g_0} u$ in terms of $\|u\|_{C^1(\Omega)}$. Thus, to obtain \eqref{53}, it suffices to obtain an upper bound for $\Delta_{g_0}u$.
\end{rmk}

In light of the second example discussed above, if one assumes $(1,0,\dots,0)\in\Gamma$ in addition to the hypotheses of Case 1 of Theorem \ref{A}, then we can take any $p>0$. In fact, we will prove the following stronger result: 

\begin{thm}\label{61}
	Suppose that, in addition to \eqref{21}--\eqref{24}, $\Gamma$ satisfies $(1,0,\dots,0)\in\Gamma$ and $\Gamma\subseteq (\Gamma_2^+)^\tau$ for some $\tau\in(0,1]$. Suppose also that $A$ and $B$ are convex in $\xi$. Then if $u\in C^4(\Omega)$ is a solution to \eqref{1} and $\Omega'\Subset\Omega$, it holds that $\|\nabla_{g_0}^2 u\|_{L^\infty(\Omega')}\leq C$, where $C$ depends only on $f, \Gamma, n,g_0,A,B,\Omega,\Omega'$ and an upper bound for $\|u\|_{C^1(\Omega)}$. Moreover, if $(f,\Gamma) = ((\sigma_2^{1/2})^\tau,(\Gamma_2^+)^\tau)$ with $\tau\in(0,1]$, the convexity assumption on $B$ can be dropped. 
\end{thm}

	\begin{rmk}
	We point out that there is no dependence on any $L^p$ norm of $\nabla_{g_0}^2 u$ in Theorem \ref{61}. Under the assumptions of this theorem, \eqref{1} is strictly elliptic: by 1-homogeneity and concavity of $f$, we have $\sum_{j=1}^n\frac{\partial f}{\partial \lambda_j}(\lambda) = f(\lambda) + \sum_{j=1}^n \frac{\partial f}{\partial\lambda_j}(\lambda)(1-\lambda_i) \geq f(1,\dots,1)>0$, and strict ellipticity then follows from \eqref{85}. 
\end{rmk}

In the special case that $(f,\Gamma) = (\sigma_2^{1/2},\Gamma_2^+)$, we are also able to prove (using different methods) a counterpart to Case 2 of Theorem \ref{A} assuming $p>3n/2$. We note that this only improves on the assumption $p>n+2$ when $n=2$ or $3$, but we state the following result in arbitrary dimension:

\begin{thm}\label{B}
	Let $\Omega$ be a bounded domain contained in a smooth Riemannian manifold $(M^n,g_0)$ of dimension $n\geq 2$. Suppose $A = A(x,z,\xi)\in C_{\operatorname{loc}}^{2}(\Omega\times\mathbb{R}\times T^*\Omega)$ is $\operatorname{Sym}(T^*\Omega\otimes T^*\Omega)$-valued and $B =B(x,z,\xi) \in C_{\operatorname{loc}}^{2}(\Omega\times\mathbb{R}\times T^*\Omega)$ is real-valued and positive. Assume $p>3n/2$. Then if $u\in C^4(\Omega)$ is a solution to \eqref{1} with $(f,\Gamma) = (\sigma_2^{1/2},\Gamma_2^+)$ and $\Omega'\Subset\Omega$, the estimate \eqref{53} holds where $C$ is a constant depending only on $n, p, g_0, A, B, \Omega, \Omega'$ and an upper bound for $\|\nabla_{g_0}^2 u\|_{L^p(\Omega)} + \|u\|_{C^1(\Omega)}$. 
\end{thm}

\begin{rmk}\label{96}
	If $g_0$ is the flat metric and $A$ is of the form $A(x,z,\xi) = A_1(x,z,\xi)\operatorname{Id}$ for some scalar function $A_1$, then by previous work of the author and Nguyen in \cite{DN21}, Theorem \ref{B} extends to the case $(f,\Gamma) = (\sigma_k^{1/k},\Gamma_k^+)$ assuming $p>(k+1)n/2$. This improves on Case 2 of Theorem \ref{A} when $k>n/2$. If one further assumes $A_1(x,z,\xi) = A_2(x,z)|\xi|^2$ for some nonnegative function $A_2$, then by \cite{DN21} we may take $p>kn/2$ when $(f,\Gamma) = (\sigma_k^{1/k},\Gamma_k^+)$. This improves on Case 1 of Theorem \ref{A} when $k>\frac{n}{2}+1$. 
\end{rmk}

We now briefly discuss our results in the context of previous work on second derivative estimates for augmented Hessian equations -- unless otherwise mentioned, any statements in the following discussion concern the Euclidean setting. As mentioned above, an example of Heinz-Lewy \cite{Sch90} demonstrates that \eqref{1} may not admit local pointwise second derivative estimates for certain choices of $A$. In the context of optimal transport, Ma, Trudinger \& Wang \cite{MTW} introduced the notion of \textit{strict regularity} for $A$ (also known as the \textit{strict MTW condition}, or \textit{strict codimension one convexity}), which imposes
\begin{equation}\label{2}
\sum_{i,j,k,l} \partial^2_{\xi_k\xi_l}A_{ij}(x,z,\xi)\zeta_i\zeta_j \eta_k\eta_l \geq a_0|\zeta|^2|\eta|^2 \quad\text{for all }\zeta,\eta\in\mathbb{R}^n \text{ s.t. }\zeta\cdot\eta=0
\end{equation} 
for some constant $a_0>0$ and all $(x,z,\xi)\in \Omega\times\mathbb{R}\times\mathbb{R}^n$. Strict regularity of $A$ is known to be sufficient to obtain local second derivative estimates and, for certain boundary value problems, global second derivative estimates -- see for instance the work of Trudinger et.~al.~\cite{MTW, JT18} in the Euclidean setting, and \cite{Guan14, GJ15, GJ16} for some related work of Guan et.~al.~ in the Riemannian setting. In fact, the weaker notion of \textit{regularity} of $A$ (also known as the \textit{MTW condition}, or \textit{codimension one convexity}), in which one allows $a_0=0$ in \eqref{2}, is sufficient to obtain global second derivative estimates for certain boundary value problems -- see e.g.~Jiang \& Trudinger \cite{JT17}. In the context of optimal transport, Loeper \cite{Loe} showed that even $C^1$ regularity may fail if $A$ is not regular.

If $A$ is assumed to be regular (but not strictly regular), then the question of whether \eqref{1} admits local second derivative estimates is less well-understood. For instance, it is still unknown whether solutions to \eqref{1} with $A\equiv 0$, $B=1$ and $(f,\Gamma) = (\sigma_k^{1/k},\Gamma_k^+)$ admit local second derivative estimates when $k=2$ and $n\geq 4$ (Warren \& Yuan \cite{WY09} gave an affirmative answer for $k=2$ when $n=3$, and Urbas \cite{Urb90} gave counterexamples for $k\geq 3$). Our assumption in Case 1 of Theorem \ref{A} that $A$ is convex in the gradient variable -- also referred to by Jiang \& Trudinger in \cite{JT20} as \textit{regularity without orthogonality} -- is a stronger assumption than regularity of $A$. However, it is easy to construct examples of $A$ which are convex in the gradient variable but not strictly regular (e.g.~any $A$ which is independent of the gradient variable). It would be interesting to determine whether Case 1 of Theorem \ref{A} still holds if $A$ is only assumed to be regular. 

\begin{rmk}
	In the papers cited above which consider the strict regularity condition on $A$, results are largely obtained using Pogorelov-type estimates, which by their nature result in second derivative estimates that depend on the $C^1$ norm of the solution and other known data, rather than any $W^{2,p}$ norms. 
\end{rmk}

An important example for which not even regularity of $A$ is satisfied is the $\sigma_k$-Yamabe equation in the so-called negative case: given a closed Riemannian manifold $(M^n,g_0)$ of dimension $n\geq 3$ and an integer $2\leq k \leq n$, one looks for a solution $g_u=e^{2u}g_0$ to 
\begin{equation}\label{37}
\sigma_k\big(\lambda(-g_u^{-1}A_{g_u})\big)  = 1, \quad \lambda(-g_u^{-1}A_{g_u})\in\Gamma_k^+
\end{equation}
on $M^n$, typically under the assumption that the background metric $g_0$ satisfies \linebreak$\lambda(-g_0^{-1}A_{g_0})\in\Gamma_k^+$ on $M^n$. Here, $A_{g_u}$ is the $(0,2)$-Schouten tensor of $g_u$, defined by 
\begin{equation*}
A_{g_u} = \frac{1}{n-2}\bigg(\operatorname{Ric}_{g_u} - \frac{R_{g_u}}{2(n-1)}g_u\bigg),
\end{equation*}
where $\operatorname{Ric}_{g_u}$ and $R_{g_u}$ denote the Ricci tensor and scalar curvature of $g_u$, respectively. The Schouten tensor arises in the Ricci decomposition of the Riemann curvature tensor and is an important quantity in conformal geometry. When $k=1$, the equation \eqref{37} reduces to the Yamabe equation in the case of negative Yamabe invariant. For $k\geq 2$, the $\sigma_k$-Yamabe equation was first studied by Viaclovsky in \cite{Via00a}. Note that \eqref{37} falls into the framework of \eqref{1}, in light of the conformal transformation law 
\begin{equation}\label{97}
- A_{g_u} = \nabla_{g_0}^2u - du\otimes du + \frac{|du|_{g_0}^2}{2}g_0 - A_{g_0}. 
\end{equation}
Whilst global \textit{a priori} $C^1$ estimates on solutions to \eqref{37} are known due to Gursky \& Viaclovsky \cite{GV03b}, neither local nor global $C^2$ estimates are known. In light of Theorems \ref{A} and \ref{B}, we have: 

\begin{cor}\label{3}
	Let $\Omega$ be a bounded domain contained in a smooth Riemannian manifold $(M^n,g_0)$ of dimension $n\geq 3$. Let $2\leq k \leq n$ be an integer, and assume $p>n+k(k-1)$ if $k\geq 3$ and $p>\min(n+2, \frac{3n}{2})$ if $k=2$. Then if $u\in C^4(\Omega)$ is a solution to \eqref{37} on $\Omega$, and $\Omega'\Subset\Omega$, it holds that
	\begin{equation*}
	\|\nabla_{g_0}^2u\|_{L^\infty(\Omega')} \leq C
	\end{equation*}
	where $C$ is a constant depending only on $k,n, p, g_0, \Omega,\Omega'$ and an upper bound for \linebreak $\|\nabla_{g_0}^2 u\|_{L^p(\Omega)} + \|u\|_{C^1(\Omega)}$. 
\end{cor}

\begin{rmk}\label{89}
	Using the uniform $C^1$ estimates of \cite{GV03b}, Li \& Nguyen \cite{LN20b} proved the existence of a Lipschitz viscosity solution to \eqref{37}. Following the argument of \cite{GV03b} (where $C^2$ estimates and existence of solutions are proved for a trace-modified version of \eqref{37}), Corollary \ref{3} implies existence of a smooth solution to \eqref{37} if $W^{2,p}$ estimates can be established for some $p>n+k(k-1)$ if $k\geq 3$ and $p>\min(n+2, \frac{3n}{2})$ if $k=2$. 
	\end{rmk} 

\begin{rmk}
	If $(M^n,g_0)$ is locally conformally flat, then by Remark \ref{96} and the change of variables $u=\log v$ in \eqref{97}, one may replace the lower bounds for $p$ in Corollary \ref{3} and Remark \ref{89} by $p>\min (n+k(k-1), \frac{(k+1)n}{2})$ for any $k\geq 2$.
\end{rmk}

\begin{rmk}
	For the positive case of the $\sigma_k$-Yamabe equation (in which $-A_{g_u}$ is replaced with $A_{g_u}$), global pointwise estimates were established in \cite{Via02} and local pointwise estimates were established in varying levels of generality in \cite{Che05, GW03b, JLL07, LL03, Li09, Wan06}, for example. For related work involving integral estimates, see for instance \cite{CGY02a, Gonz05, Gonz06, Han04}.
\end{rmk}

We now briefly remark on the methods used to prove Theorems \ref{A} and \ref{B}. Our method for proving Theorem \ref{A} is similar to that used in \cite{D22}, and is based on an application of the Alexandrov-Bakelman-Pucci (ABP) estimate to produce an upper bound for $\Delta_{g_0} u$ in terms of certain integral quantities. This method is inspired by previous work of Bao et.~al. \cite{BCGJ03} on Hessian quotient equations on Euclidean domains (see also \cite{BL}), suitably extended to deal with third order terms arising from $A$, more general operators $f$ and the presence of a non-Euclidean metric. On the other hand, our method for proving Theorem \ref{B} is similar to that used in \cite{DN21}, and is based on an integrability improvement argument followed by Moser iteration. Similar ideas were used by Urbas \cite{Urb00, Urb01} in the context of $k$-Hessian equations on Euclidean domains. In comparison with the proof of Theorem \ref{A}, it is now both the term $A$ \textit{and} the linearised operator $F_i^j$ which produce new third order terms in our estimates (in the $k$-Hessian case, $F_i^j$ is divergence-free). However, by exploiting a divergence structure in the case $k=2$, we are able to produce a cancellation phenomenon to deal with these third order terms. This is somehow more delicate than the method used to prove Theorem \ref{A}, where any third order terms are estimated more directly using properties of the concave envelope of a suitable function, and terms involving $F$ are ultimately dealt with using the condition \eqref{26} (rather than using any divergence structure to produce cancellations). We note than in both \cite{D22} and \cite{DN21}, we obtained local $C^{1,1}$ estimates on $W^{2,p}$-strong solutions to a more restrictive class of equations than \eqref{1}, and our estimates applied only in the Euclidean setting. In this paper we are concerned only with \textit{a priori} estimates, which allows us to address a much broader class of equations, including on Riemannian manifolds.

The plan of the paper is as follows. In Section \ref{A} we prove Theorem \ref{A}, starting with a proof of the first case in Section \ref{4}. The proof of the second case of Theorem \ref{A} is given in Section \ref{5}, and requires only minor amendments to the proof of the first case. In Section \ref{88}, we prove the last statement in Theorem \ref{A}. In Section \ref{90} we prove Theorem \ref{61}, and in Section \ref{77} we give the proof of Theorem \ref{B}. \medskip

\noindent\textbf{Acknowledgements:} The author would like to thank Luc Nguyen and Yannick Sire for helpful comments on the introduction to this paper. A preliminary version of Theorem \ref{B} was obtained whilst the author was supported by EPSRC grant EP/L015811/1.

\section{Proof of Theorem \ref{A}} 

In this section we prove Theorem \ref{A}. We start in Section \ref{4} by addressing Case 1 in Theorem \ref{A}, namely we assume that $A$ and $B$ are convex in $\xi$ and $p>\gamma$. In Section \ref{5} we then prove Case 2 in Theorem \ref{A}, where we drop the convexity assumptions on $A$ and $B$ but impose $p>n+\gamma$; this requires only minor amendments to the arguments in Case 1. In Section \ref{88}, we explain how the convexity assumption on $B$ in Case 1 of Theorem \ref{A} can be removed if $(f,\Gamma) = ((\sigma_2^{1/2})^\tau, (\Gamma_2^+)^\tau)$ for some $\tau\in(0,1]$. 

First, some remarks on our notation and conventions for this section are in order:  \medskip

\noindent $\bullet$ Since our desired estimates are local, it suffices (by a standard covering argument) to consider the case that $\Omega=B_{2R}$ and $\Omega'=B_{R}$ are Euclidean balls centred at the origin in $\mathbb{R}^n$, equipped with a Riemannian metric $g_0$.  \medskip 

\noindent$\bullet$ $\nabla $, $\Delta $ and $\nabla^2 $ will denote the gradient, Laplacian and Hessian with respect to $g_0$, respectively. We instead write $D$ and $D^2$ for the gradient and Hessian with respect to the flat metric on $\mathbb{R}^n$, respectively. $\nabla_i$ will denote the covariant derivative with respect to $g_0$ in the $i$'th coordinate direction, and $\partial_i$ will denote the Euclidean partial derivative in the $i$'th coordinate direction. \medskip 

\noindent $\bullet$ $|\cdot|_{g_0}$ will indicate a norm taken with respect to $g_0$, and $|\cdot|$ (without any subscript) will indicate a norm taken using the standard Euclidean inner product. \medskip 

\noindent$\bullet$ We write $dx$ for the volume form of the Euclidean metric, and $dv_{g_0} = \sqrt{\operatorname{det}((g_0)_{ij})}\,dx$ for the volume form of the metric $g_0$. \medskip 

\noindent$\bullet$ We denote by $C$ any constant depending on $f, \Gamma, n, p, g_0, A, B, R$ and an upper bound for $\|u\|_{C^1(B_{2R})} + \|\nabla^2 u\|_{L^\gamma(B_{2R},g_0)}$; the dependence on $\|\nabla^2_{g_0}u\|_{L^p(B_{2R},g_0)}$ in our final estimate will be explicit. The value of $C$ may change from line to line. \medskip 

\noindent$\bullet$ We use Einstein summation convention whenever an index appears in an upper and lower position. Indices will be implicitly raised and lowered in calculations using the metric $g_0$ and its inverse, e.g.~$A[u]_{ij} = (g_0)_{ia}A[u]_j^a$.

\subsection{Proof of Case 1 in Theorem \ref{A}}\label{4}

~\medskip 

\noindent Let us briefly explain the main components of the proof. We start by defining
\begin{equation*}
\eta(x) = \bigg(1 - \frac{|x|^2}{4R^2}\bigg)^{\beta}
\end{equation*}
and $v=\eta\Delta u$, where $\beta>2$ is a constant to be determined and $|x|$ is the Euclidean distance from $x\in B_{2R}$ to the origin (recall that $B_{2R}$ now denotes the Euclidean ball of radius $2R$ centred at the origin). Recall that, if we denote $W[u] = \nabla_{g_0}^2 u - A[u]$ (considered as a $(1,1)$-tensor), then the linearised operator 
\begin{equation*}
F_i^j = \frac{\partial f(\lambda(W[u]))}{\partial (W[u])_j^i}
\end{equation*}
is positive definite in $B_{2R}$ by ellipticity. Our first step is to obtain an upper bound for $-F^{ij}\partial_i\partial_j v$ in terms of $\operatorname{tr}(F)|D\Delta u|$, $\operatorname{tr}(F)|\Delta u|$ and lower order terms. More precisely, we prove: 
\begin{lem}\label{7}
	Let $g_0$ be a Riemannian metric on $\Omega = B_{2R}$ and suppose $f,\Gamma, A$ and $B$ are as in Case 1 of Theorem \ref{A} (with $(f,\Gamma)$ not necessarily satisfying \eqref{26}). Then for any solution $u\in C^4(B_{2R})$ to \eqref{1}, it holds that 
	\begin{align}\label{36}
	-F^{ij}\partial_i\partial_j v & \leq C\operatorname{tr}(F)\bigg(  \eta\Big(|D\Delta u| + |\nabla^2 u| + 1 \Big) + |D \eta||D\Delta u|   + \big(|D^2\eta| + |D \eta|\big)|\Delta u|\bigg) \nonumber \\
	& \qquad + C\eta\Big(|D\Delta u|+ |\nabla^2 u| + 1 \Big).
	\end{align}
\end{lem}

The next step is to obtain a bound for the third order terms in \eqref{36}. We will see that such an estimate is only required to hold on $\Gamma_v^+(B_{2R})$, defined to be the upper contact set of $v$ in $B_{2R}$ with respect to the Euclidean structure:
\begin{equation*}
\Gamma_v^+(B_{2R}) = \{x\in B_{2R}:v(z) \leq v(x) + \nu\cdot (z-x) \text{ for all }z\in B_{2R}, \text{ for some }\nu\in\mathbb{R}^n\}. 
\end{equation*}
Indeed, following \cite{BCGJ03} we have:
\begin{lem}\label{50}
	On $\Gamma_v^+(B_{2R})$,
	\begin{equation}\label{6}
	\eta |D\Delta u| \leq (1+\beta)R^{-1}\eta^{-\frac{1}{\beta}}v. 
	\end{equation}
\end{lem}
After substituting \eqref{6} back into \eqref{36}, applying the ABP estimate and carrying out some routine computations, we will obtain an upper bound for $\Delta u$ on $B_R$, from which the desired estimate follows (see Remark \ref{63}). For later reference, we state the ABP estimate here in the form that we will use it:

\begin{thm}[see e.g.~{\cite[Chapter 9]{GT}}]\label{78}
	Suppose $a^{ij}$ is smooth and positive definite on a smooth bounded domain $U\subset\mathbb{R}^n$. Then there exists a constant $C=C(n)$ such that for any $\phi\in C^2(U)\cap C^0(\overline{U})$ with $\phi\equiv 0$ on $\partial U$, one has
	\begin{align*}
	\sup_U \phi \leq Cd\bigg(\int_{\Gamma_\phi^+(U)}\frac{(-a^{ij}\partial_i\partial_j\phi)^n}{\operatorname{det}(a^{ij})}\,dx\bigg)^{1/n},
	\end{align*}
	where $d$ denotes the diameter of $U$. 
\end{thm}

We now give the proof of Lemma \ref{7}:

\begin{proof}[Proof of Lemma \ref{7}]
	
By appealing to the equation \eqref{1} and the concavity of $f$, we first observe that
\begin{equation}\label{38}
\Delta B[u] = \Delta f(\lambda(\nabla_{g_0}^2 u - A[u])) \leq F^{ij}\nabla^k\nabla_k \big(\nabla_i\nabla_j  u - A[u]_{ij}\big). 
\end{equation}
By standard formulas for commuting derivatives, the fact that $|F|_{g_0} \leq C\operatorname{tr}(F)$ (which holds since $F$ is positive definite) and the fact that we allow our constants to depend on $\|u\|_{C^1(B_{2R})}$, we have
\begin{align}
F^{ij}\nabla^k\nabla_k \nabla_i\nabla_j u \leq F^{ij}\nabla_i\nabla_j \Delta u + C\operatorname{tr}(F)|\nabla^2 u|_{g_0} 
\end{align}
and hence
\begin{align*}
\Delta B[u] \leq F^{ij}\nabla_i\nabla_j \Delta u + C\operatorname{tr}(F)|\nabla^2 u|_{g_0} - F^{ij}\Delta(A[u])_{ij}. 
\end{align*}
Therefore, using the fact that $v=\eta\Delta u$, we obtain
\begin{align}\label{25}
F^{ij}\nabla_i\nabla_j v & \geq \eta F^{ij}\Delta(A[u])_{ij} + \eta\Delta B[u] - C\eta\operatorname{tr}(F)|\nabla^2 u|_{g_0} + 2F^{ij}\nabla_i\eta\nabla_j \Delta u \nonumber \\
& \quad + \Delta u F^{ij}\nabla_i\nabla_j\eta. 
\end{align}

We now consider the term $\Delta(A[u])_{ij}$. By the chain rule,
\begin{align*}
\nabla_k \big(A[u]\big)_{ij} = \frac{\partial A_{ij}}{\partial x_k}[u] + \frac{\partial A_{ij}}{\partial z}[u]\nabla_k u + \frac{\partial A_{ij}}{\partial\xi_l}[u]\nabla_k\nabla_l u
\end{align*}
and therefore
\begin{align}\label{20}
\nabla^k&\nabla_k \big(A[u]\big)_{ij} \nonumber \\
& = \frac{\partial^2 A_{ij}}{\partial x_k^2}[u] + \frac{\partial^2 A_{ij}}{\partial z \partial x_k}[u]\nabla^k u + \frac{\partial^2 A_{ij}}{\partial\xi_l\partial x_k}[u]\nabla^k\nabla_l u \nonumber \\
& \quad + \bigg(\frac{\partial^2 A_{ij}}{\partial x_k\partial z}[u] + \frac{\partial^2 A_{ij}}{\partial z^2}[u]\nabla^k u + \frac{\partial^2 A_{ij}}{\partial\xi_l\partial z}[u]\nabla^k\nabla_l u\bigg)\nabla_k u +  \frac{\partial A_{ij}}{\partial z}[u]\Delta u \nonumber \\
& \quad + \bigg(\frac{\partial^2 A_{ij}}{\partial x_k\partial \xi_l}[u] +\frac{\partial^2 A_{ij}}{\partial z\partial \xi_l}[u]\nabla^k u + \frac{\partial^2 A_{ij}}{\partial\xi_p\partial\xi_l}[u]\nabla^k\nabla_p u\bigg)\nabla_k\nabla_l u + \frac{\partial A_{ij}}{\partial \xi_l}\nabla^k\nabla_k\nabla_l u.
\end{align}
Now, by positivity of $F$ and the assumption that $A$ is convex in the gradient variable, it holds that
\begin{align}\label{27}
F^{ij}\frac{\partial^2 A_{ij}}{\partial\xi_p\partial\xi_l}[u]\nabla^k\nabla_p u\nabla_k\nabla_l u \geq 0. 
\end{align}
After commuting derivatives in the final term in \eqref{20}, we therefore obtain from \eqref{20} and \eqref{27} the estimate
\begin{align}\label{30}
F^{ij}\Delta(A[u])_{ij} \geq -C\operatorname{tr}(F)\Big(|\nabla \Delta u|_{g_0} + |\nabla^2 u|_{g_0} + 1\Big),
\end{align}
where $C$ now depends on the $C^{2}$ norm of $A$. By identical reasoning, the assumption that $B$ is convex with respect to the gradient variable also implies
\begin{equation}\label{32}
\Delta B[u] \geq -C\Big(|\nabla\Delta u|_{g_0} + |\nabla^2 u|_{g_0} + 1\Big).
\end{equation}

Substituting \eqref{30} and \eqref{32} back into \eqref{25}, we therefore see that
\begin{align}\label{33}
-F^{ij}\nabla_i\nabla_j v & \leq C\operatorname{tr}(F)\bigg(\eta\Big(|\nabla\Delta u|_{g_0} + |\nabla^2 u|_{g_0} + 1 \Big) + |\nabla \eta|_{g_0}|\nabla\Delta u|_{g_0} + |\nabla^2\eta|_{g_0}|\Delta u|\bigg) \nonumber \\
& \qquad + C\eta \Big(|\nabla\Delta u|_{g_0} + |\nabla^2 u|_{g_0} + 1 \Big). 
\end{align} 
Now, given any function $\phi$ on $B_{2R}$, we have
\begin{align*}
|\nabla \phi|_{g_0}^2 = g_0^{ij}\partial_i\phi\partial_j\phi \leq C(g_0)\delta^{ij}\partial_i\phi \partial_j\phi = C(g_0)|D\phi|^2,
\end{align*}
\begin{align*}
|\nabla^2\phi|_{g_0}^2 & = g_0^{ia}g_0^{jb}\nabla_a\nabla_b \phi\nabla_i\nabla_j\phi \nonumber \\
& = g_0^{ia}g_0^{jb}\big(\partial_a\partial_b\phi - \Gamma_{ab}^c\partial_c\phi\big)\big(\partial_i\partial_j\phi - \Gamma_{ij}^k\partial_k\phi\big) \nonumber \\
& \leq C(g_0)\big(|D^2\phi|^2+ |D \phi|^2\big),
\end{align*}
and likewise
\begin{align*}
|\nabla^2\phi|_{g_0}^2 \leq C(g_0)|\nabla^2 \phi|^2,
\end{align*}
where $\Gamma_{ij}^k$ are the Christoffel symbols of the metric $g_0$. Use these three estimates in \eqref{33} then yields
\begin{align}\label{91}
-F^{ij}\nabla_i\nabla_j v & \leq C\operatorname{tr}(F)\bigg(\eta\Big(|D\Delta u|+ |\nabla^2 u| + 1 \Big) + |D \eta||D\Delta u|  + \big(|D^2\eta| + |D \eta|\big)|\Delta u|\bigg) \nonumber \\
& \qquad + C\eta \Big(|D\Delta u|+ |\nabla^2 u| + 1 \Big).
\end{align}
The desired estimate \eqref{36} then follows from \eqref{91} after observing
\begin{align*}
F^{ij}\nabla_i\nabla_j v  & = F^{ij}\big(\partial_i\partial_j v - \Gamma_{ij}^k\partial_k v\big) \nonumber \\
&  \leq F^{ij}\partial_i\partial_j v + C\operatorname{tr}(F)|D v| \nonumber \\
& \leq F^{ij}\partial_i\partial_j v + C\operatorname{tr}(F)\big(\eta|D\Delta u| + |D\eta||\Delta u|\big).
\end{align*}
\end{proof} 

For later reference, we note here the following estimates on the derivatives of $\eta$ (which follow immediately from the definition of $\eta$):
\begin{equation}\label{10}
|D\eta| \leq CR^{-1}\eta^{1-\frac{1}{\beta}} \quad\text{and}\quad |D^2 \eta| \leq CR^{-2}\eta^{1-\frac{2}{\beta}}.
\end{equation}

\noindent We now give the proof of Lemma \ref{50}, which is essentially the same as that given in \cite{BCGJ03} in the Euclidean setting:

\begin{proof}[Proof of Lemma \ref{50}]
	For $x\in\Gamma_v^+(B_{2R})$ such that $Dv(x)\not=0$, let $z\in \partial B_{2R}$ be such that
	\begin{equation*}
	\frac{z-x}{|z-x|}  = - \frac{D v(x)}{|D v(x)|}. 
	\end{equation*}
	Since $v=0$ on $\partial B_{2R}$ and $|z-x| \geq |z| -|x| = 2R - |x| \geq R\eta^\frac{1}{\beta}$ (the last inequality following from the definition of $\eta$), we thus have for such points $x\in \Gamma_v^+(B_{2R})$ that 
	\begin{equation}\label{t49}
	v(x) \geq v(z) -  Dv(x)\cdot(z-x)  = - D v(x)\cdot(z-x)\ = |z-x||D v(x)|\geq R\eta^{\frac{1}{\beta}}|D v(x)|,
	\end{equation}
	where $\cdot$ denotes the Euclidean inner product. At such points $x\in\Gamma_v^+(B_{2R})$ we therefore have
	\begin{align}\label{e5}
	\eta|D\Delta u|  = |D v - \Delta uD \eta|  \leq |D v|+ \Delta u|D \eta| \,\, \leftstackrel{\eqref{t49}}{\leq} \frac{v}{R\eta^\frac{1}{\beta}} + \frac{v}{\eta} \frac{\beta}{R}\eta^{1-\frac{1}{\beta}}  = \frac{(1+\beta)v}{R\eta^{\frac{1}{\beta}}}.
	\end{align}
	Note that at points $x\in\Gamma_v^+(B_{2R})$ where $Dv(x)=0$, it is clear that \eqref{e5} still holds. 
\end{proof}

We now complete the proof of Case 1 in Theorem \ref{A}:

\begin{proof}[Proof of Case 1 in Theorem \ref{A}]
	As explained at the start of Section 2, it suffices to consider the case that $\Omega = B_{2R}$ is a Euclidean ball of radius $2R$ centred at the origin and $\Omega' = B_R$. 
	
	Since $\Gamma\subseteq(\Gamma_2^+)^\tau$ for some $\tau\in(0,1]$, there exists a constant $C$ depending on $\|u\|_{C^1(B_{2R})}$ such that $|\nabla^2 u| \leq C(|\Delta u|+1)$ (see Remark \ref{63}). Using this fact after substituting the estimate \eqref{6} of Lemma \ref{50} back into the estimate \eqref{36} of Lemma \ref{7}, we therefore obtain on $\Gamma_v^+(B_{2R})$ 
	\begin{align}\label{11}
	0 \leq -F^{ij}\partial_i\partial_j v & \leq C\operatorname{tr}(F)\bigg(\frac{(1+\beta)v}{R\eta^{1/\beta}} + v + \eta + \frac{(1+\beta)v}{R\eta^{1+\frac{1}{\beta}}}|D\eta| + \frac{|D^2\eta|+ |D\eta|}{\eta}v\bigg) \nonumber \\
	& \qquad + C\bigg(\frac{(1+\beta)v}{R\eta^{1/\beta}} + v + \eta\bigg). 
	\end{align}
	Note we have used $D^2 v \leq 0$ on $\Gamma_v^+(B_{2R})$ to assert $0 \leq -F^{ij}\partial_i\partial_j v $. Next we appeal to the estimates in \eqref{10} for $|D\eta|$ and $|D^2\eta|$, which when substituted into \eqref{11} yields
	\begin{align}\label{42}
	0 \leq -F^{ij}\partial_i\partial_j v & \leq C\operatorname{tr}(F)\bigg(\frac{v}{R\eta^{\frac{1}{\beta}}} + v + \eta + \frac{v}{R^2 \eta^{\frac{2}{\beta}}} \bigg) + C\bigg(\frac{v}{R\eta^{\frac{1}{\beta}}} + v + \eta \bigg) 
	\end{align}
	on $\Gamma_v^+(B_{2R})$. We now proceed in a similar way to \cite{D22}. By \eqref{42} and the assumption \eqref{26}, we have on $\Gamma_v^+(B_{2R})$
\begin{align}\label{e9}
0   \leq \frac{-F^{ij}\partial_i\partial_j v}{(\det F^{ij})^{1/n}} & \leq \frac{C\operatorname{tr}(F)}{(\operatorname{det}F^{ij})^{1/n}}\bigg(\frac{v}{R\eta^{\frac{1}{\beta}}} +  v +  \eta +   \frac{v}{R^2\eta^{\frac{2}{\beta}}}\bigg) + \frac{C}{(\operatorname{det}F^{ij})^{1/n}}\bigg(\frac{v}{R\eta^{\frac{1}{\beta}}} +  v +  \eta\bigg)\nonumber \\
&  \leq C\bigg(\frac{\sigma_1(\lambda(\nabla^2 u - A[u]))}{f(\lambda(\nabla^2 u - A[u]))}\bigg)^{\gamma/n}\bigg( \frac{v}{R\eta^{\frac{1}{\beta}}} + v + \eta  +  \frac{v}{R^2\eta^{\frac{2}{\beta}}}\bigg) \nonumber \\
& \quad + C\bigg(\frac{\sigma_1(\lambda(\nabla^2 u - A[u]))}{f(\lambda(\nabla^2 u - A[u]))}\bigg)^{\gamma/n}\frac{1}{\operatorname{tr}(F)}\bigg(\frac{v}{R\eta^{\frac{1}{\beta}}} +  v +  \eta\bigg).
\end{align}
Now, since $B[u]$ and $\operatorname{tr}(F(A))= \sum_i f_{\lambda_i}(\lambda(A)) = f(\lambda(A)) +\sum_i f_{\lambda_i}(\lambda(A))(1-\lambda_i) \geq f(1,\dots,1)$ are both bounded away from zero, and since $\sigma_1(\lambda(\nabla^2 u - A[u])) \leq \Delta u + C$, \eqref{e9} implies (after using the equation \eqref{1}) that 
\begin{align}\label{e10}
0   \leq \frac{-F^{ij}\partial_i\partial_j v}{(\det F^{ij})^{1/n}} \leq C\bigg( \frac{v}{R\eta^{\frac{1}{\beta}}} + v + \eta  +  \frac{v}{R^2\eta^{\frac{2}{\beta}}}\bigg)(\Delta u + C)^{\gamma/n} \quad\text{on }\Gamma_v^+(B_{2R}).
\end{align}

We next apply the ABP estimate as stated in Theorem \ref{78}, which in combination with \eqref{e10} yields 
\begin{align}\label{2-}
&\sup_{B_{2R}} v  \leq CR\bigg(\int_{\Gamma_v^+(B_{2R})}\frac{(-F^{ij}\partial_i\partial_j v)^n}{\operatorname{det}(F^{ij})}\,dx\bigg)^{1/n} \nonumber \\
& \leq C\bigg(\int_{\Gamma_v^+(B_{2R})}(\eta^{-\frac{1}{\beta}}v)^n(\Delta u + C)^{\gamma}\,dv_{g_0}\bigg)^{1/n} +CR\bigg(\int_{\Gamma_v^+(B_{2R})} v^n(\Delta u + C)^{\gamma}\,dv_{g_0}\bigg)^{1/n} \nonumber \\
& \quad  +CR\bigg(\int_{\Gamma_v^+(B_{2R})} \eta^n(\Delta u + C)^{\gamma}\,dv_{g_0}\bigg)^{1/n}+ CR^{-1}\bigg(\int_{\Gamma_v^+(B_{2R})}(\eta^{-\frac{2}{\beta}}v)^n(\Delta u + C)^{\gamma}\,dv_{g_0}\bigg)^{1/n}.
\end{align}
Note that we have used $dx = (\operatorname{det}((g_0)_{ij}))^{-1/2}\,dv_{g_0}$ and the fact that $(\operatorname{det}((g_0)_{ij}))^{-1/2}$ is bounded by a constant depending only on $g_0$ in order to replace $dx$ with $dv_{g_0}$ in \eqref{2-}. For the remainder of the proof, all integrals are implicitly assumed to be with respect to $dv_{g_0}$. 

We estimate each of the four integrals on the RHS of \eqref{2-} in turn, starting with the last one. Writing $\eta^{-\frac{2}{\beta}}v = v^{1-\frac{2}{\beta}}(\Delta u)^{\frac{2}{\beta}}$, and noting that $1-\frac{2}{\beta}>0$ (since $\beta>2$), we see
\begin{align*}
\bigg(\int_{\Gamma_v^+(B_{2R})}\!(\eta^{-\frac{2}{\beta}}v)^n(\Delta u+C)^{\gamma}\!\bigg)^{1/n} \!\leq (\sup_{B_{2R}}v)^{1-\frac{2}{\beta}}\bigg(\int_{B_{2R}}|\Delta u|^{\frac{2n}{\beta}}(\Delta u+C)^{\gamma}\!\bigg)^{1/n}\!\!,
\end{align*}
where we have assumed $\sup_{B_{2R}} v \geq 0$ (otherwise we are done). For the first term on the RHS of \eqref{2-}, note $\eta^{-\frac{1}{\beta}}v = v^{1-\frac{2}{\beta}}(\sqrt{\eta}|\Delta u|)^{\frac{2}{\beta}} \leq v^{1-\frac{2}{\beta}}|\Delta u|^{\frac{2}{\beta}}$, so
\begin{align*}
\bigg(\int_{\Gamma_v^+(B_{2R})}(\eta^{-\frac{1}{\beta}}v)^n(\Delta u+C)^{\gamma}\bigg)^{1/n} \leq (\sup_{B_{2R}}v)^{1-\frac{2}{\beta}}\bigg(\int_{B_{2R}}|\Delta u|^{\frac{2n}{\beta}}(\Delta u+C)^{\gamma}\bigg)^{1/n}.
\end{align*}
Similarly, for the second term, $v=v^{1-\frac{2}{\beta}}(\eta|\Delta u|)^{\frac{2}{\beta}} \leq v^{1-\frac{2}{\beta}}|\Delta u|^{\frac{2}{\beta}}$, so 
\begin{align*}
\bigg(\int_{\Gamma_v^+(B_{2R})} v^n(\Delta u+C)^{\gamma}\bigg)^{1/n}  \leq (\sup_{B_{2R}}v)^{1-\frac{2}{\beta}}\bigg(\int_{B_{2R}}|\Delta u|^{\frac{2n}{\beta}}(\Delta u+C)^{\gamma}\bigg)^{1/n},
\end{align*}
and for the third term on the RHS of \eqref{2-} it is easy to see that
\begin{equation*}
\bigg(\int_{\Gamma_v^+(B_{2R})} \eta^n(\Delta u+C)^{\gamma}\bigg)^{1/n} \leq \bigg(\int_{B_{2R}} (\Delta u+C)^{\gamma}\bigg)^{1/n} \leq C. 
\end{equation*}

Substituting the previous four estimates into \eqref{2-} and using the fact that $R$ is bounded, we therefore obtain
\begin{align}\label{30'}
\sup_{B_{2R}} v \leq CR^{-1}(\sup_{B_{2R}}v)^{1-\frac{2}{\beta}}\bigg(\int_{B_{2R}}|\Delta u|^{\frac{2n}{\beta}}(\Delta u+C)^{\gamma}\bigg)^{1/n} + CR^{-1}.
\end{align}

We now choose $\beta = \frac{2n}{p-\gamma}$, where $p$ is as in the statement of Case 1 in Theorem \ref{A}, so that $\frac{2n}{\beta} = p-\gamma$. After applying H\"older's inequality to the integral on the RHS of \eqref{30'} and dividing through by $(\sup_{B_{2R}}v)^{1-\frac{2}{\beta}}$, we obtain the estimate
\begin{align}\label{43}
(\sup_{B_{2R}} v)^{2/\beta}\leq CR^{-1}\|\Delta u+C \|_{L^p(B_{2R})}^{p/n} + \frac{CR^{-1}}{(\sup_{B_{2R}} v)^{1-\frac{2}{\beta}}}.
\end{align}
If $\sup_{B_{2R}} v\leq 1$ then we are done. Supposing otherwise, \eqref{43} then implies
\begin{align*}
(\sup_{B_{2R}} v)^{2/\beta}\leq CR^{-1}\big(1+\|\Delta u +C\|_{L^p(B_{2R})}^{p/n} \big),
\end{align*}
and we therefore arrive at the estimate 
\begin{equation*}
\sup_{B_R} \Delta u \leq CR^{-\beta/2}\big(1+\|\Delta u + C \|_{L^p(B_{2R})}^{p/n} \big)^{\beta/2}.
\end{equation*}
As explained in Remark \ref{63}, the estimate for $\|\nabla^2 u\|_{L^\infty(B_{R})}$  then follows. 
\end{proof}

\subsection{Proof of Case 2 in Theorem \ref{A}}\label{5}

~\medskip 

\noindent We now prove the second case in Theorem \ref{A}, which requires only minor changes to the proof in the previous section. As before, $B_{2R}$ denotes the Euclidean ball of radius $2R$ centred at the origin. We first observe the following counterpart to Lemma \ref{7}:

\begin{lem}\label{7'}
	Let $g_0$ be a Riemannian metric on $\Omega=B_{2R}$ and suppose $f,\Gamma, A$ and $B$ are as in Case 2 of Theorem \ref{A} (with $(f,\Gamma)$ not necessarily satisfying \eqref{26}). Then for any solution $u\in C^4(B_{2R})$ to \eqref{1}, it holds that 
	\begin{align}\label{36'}
	-F^{ij}\partial_i\partial_j v & \leq C\operatorname{tr}(F)\bigg(  \eta\Big(|D\Delta u| + |\nabla^2 u|^2 + 1 \Big) + |D \eta||D\Delta u|   + \big(|D^2\eta| + |D \eta|\big)|\Delta u|\bigg) \nonumber \\
	& \qquad + C\eta\Big(|D\Delta u|+ |\nabla^2 u|^2 + 1 \Big).
	\end{align}
\end{lem}

\begin{rmk}
	We point out that the only difference between \eqref{36'} and \eqref{36} is in the exponent of the $|\nabla^2 u|$ terms. 
\end{rmk}

\begin{proof}
	The computations \eqref{38}--\eqref{20} in the proof of Lemma \ref{7} still apply under the hypotheses of Lemma \ref{7'}. However, the penultimate term in \eqref{20} can no longer be dropped (since \eqref{27} is no longer necessarily true), and hence \eqref{30} is replaced by the weaker estimate
	\begin{align*}
	F^{ij}\Delta(A[u])_{ij} \geq -C\operatorname{tr}(F)\Big(|\nabla \Delta u|_{g_0} + |\nabla^2 u|^2_{g_0} + 1\Big).
	\end{align*}
	Likewise, \eqref{32} is replaced by the weaker estimate
	\begin{equation}\label{92}
	\Delta B[u] \geq -C\Big(|\nabla\Delta u|_{g_0} + |\nabla^2 u|^2_{g_0} + 1\Big).
	\end{equation}
	The proof of \eqref{36'} then proceeds exactly as in the proof of Lemma \ref{7}. 
\end{proof}

\begin{proof}[Proof of Case 2 in Theorem \ref{A}]
	In the same way that we obtained \eqref{42} from \eqref{36}, we obtain from \eqref{36'} the estimate 
	\begin{align}\label{42'}
	0 \leq -F^{ij}\partial_i\partial_j v & \leq C\operatorname{tr}(F)\bigg(\frac{v}{R\eta^{\frac{1}{\beta}}} + v|\Delta u| + \eta + \frac{v}{R^2 \eta^{\frac{2}{\beta}}} \bigg) + C\bigg(\frac{v}{R\eta^{\frac{1}{\beta}}} + v|\Delta u| + \eta \bigg)
	\end{align}
	on $\Gamma_v^+(B_{2R})$. We point out that the only difference between \eqref{42'} and \eqref{42} are the $v|\Delta u|$ terms. Keeping track of these terms, we see that the estimate \eqref{2-} remains the same except that the second term on the middle line should be replaced with
	\begin{equation}\label{71}
	CR\bigg(\int_{\Gamma_v^+(B_{2R})}(v|\Delta u|)^n(\Delta u + C)^\gamma\bigg)^{1/n}. 
	\end{equation}
	Assuming once again that $\sup_{B_{2R}}v \geq 0$ (otherwise we are done), we see that the expression in \eqref{71} is bounded from above by 
	\begin{equation*}
	CR(\sup_{B_{2R}})^{1-\frac{2}{\beta}}\bigg(\int_{B_{2R}}|\Delta u|^{n+\frac{2n}{\beta}}(\Delta u + C)^\gamma\bigg)^{1/n}.
	\end{equation*}
	The proof then proceeds in exactly the same way to Case 1, now choosing $\beta = \frac{2n}{p-n -\gamma}$ where $p$ is as in the statement of Case 2. 
	\end{proof}

\subsection{The case $(f,\Gamma) = ((\sigma_2^{1/2})^\tau, (\Gamma_2^+)^\tau)$}\label{88}

~\medskip 

\noindent To complete the proof of Theorem \ref{A}, it remains to show that the convexity assumption on $B$ in Case 1 can be dropped if one assumes $(f,\Gamma) = ((\sigma_2^{1/2})^\tau, (\Gamma_2^+)^\tau)$ for some $\tau\in(0,1]$. 

\begin{proof}[Proof of the last statement in Theorem \ref{A}]
It suffices to show that Lemma \ref{7} still holds, since this is the only place in the proof of Case 1 where a convexity assumption on $B$ is used. It is well-known that the linearisation of the $\sigma_2$-operator is given by the first Newton tensor, that is
\begin{equation}\label{95}
\frac{\partial \sigma_2^{1/2}(\lambda(X))}{\partial X^i_j} = \frac{1}{2\sigma_2^{1/2}(\lambda(X))}\frac{\partial \sigma_2(\lambda(X))}{\partial X^i_j} = \frac{1}{2\sigma_2^{1/2}(\lambda(X))}\big(\sigma_1(\lambda(X))\delta_i^j - X_i^j\big). 
\end{equation}
Therefore, in the case that $u$ is a solution to \eqref{1} with $(f,\Gamma) = ((\sigma_2^{1/2})^\tau, (\Gamma_2^+)^\tau)$ for some $\tau\in(0,1]$, 
\begin{align*}
\operatorname{tr}(F) & = \frac{n-1}{2B[u]}\operatorname{tr}\big[\tau (\nabla^2 u - A[u]) + (1-\tau)\operatorname{tr}(\nabla^2 u - A[u])\operatorname{Id}\big] > C^{-1}(\Delta u - \operatorname{tr}(A[u])). 
\end{align*}
Since $|\nabla^2 u|_{g_0} \leq C(\Delta u + 1)$, it follows that $|\nabla^2 u|_{g_0} \leq C(\operatorname{tr}(F) + 1)$. Substituting this into the estimate \eqref{92}, we therefore see that
\begin{align}\label{93}
\Delta B[u] \geq - C\Big(|\nabla\Delta u|_{g_0} + |\nabla^2 u|_{g_0} + 1\Big) - C\operatorname{tr}(F)|\nabla^2 u|_{g_0}.
\end{align}
With \eqref{93} in place of \eqref{30}, one still obtains the estimate \eqref{33}, and the proof of Lemma \ref{7} then proceeds as before. 
\end{proof}

\section{Proof of Theorem \ref{61}}\label{90}

\noindent In this short section we prove Theorem \ref{61}, which states that we can remove the dependence on any $L^p$ norm of $\nabla_{g_0}^2 u$ when $A$ and $B$ are convex in the gradient variable and $(1,0,\dots,0)\in\Gamma$. 

\begin{proof}[Proof of Theorem \ref{61}]
By a standard covering argument, it suffices to prove Theorem \ref{61} in the case that $\Omega = B_{3R}$ is a sufficiently small geodesic ball and $\Omega' = B_{R}$. First note that by Case 1 of Theorem \ref{A}, under our current hypotheses we have
\begin{equation}\label{94}
\|\nabla_{g_0}^2 u\|_{L^\infty(B_R)} \leq C
\end{equation}
where $C$ depends on an upper bound $\|u\|_{C^1(B_{2R})} + \|\nabla_{g_0}^2 u\|_{L^1(B_{2R})}$ and other given data.

Now suppose that $R$ is sufficiently small so that there exists a strictly convex function $\phi$ on $B_{3R}$. In light of the second sentence in Remark \ref{63}, $u + C_1\phi$ is then subharmonic on $B_{3R}$ for a sufficiently large constant $C_1$ depending on $g_0$ and an upper bound for $\|u\|_{C^1(B_{3R})}$. To prove the theorem, by \eqref{94} it suffices to obtain an estimate for $ \|\nabla_{g_0}^2 u\|_{L^1(B_{2R})}$ in terms of $\|u\|_{C^1(B_{3R})}$ and other given data on $B_{3R}$. 

To this end, note that if $\psi$ is any subharmonic function on $B_{3R}$, then $\Delta_{g_0}\psi$ satisfies an $L^1$ estimate on $B_{2R}$ in terms of $\|\psi\|_{L^1(B_{3R})}$. Indeed, taking $\eta\in C^\infty_c(B_{3R})$ to be a nonnegative cutoff function satisfying $\eta\equiv 1$ on $B_{2R}$, $|\nabla_{g_0}\eta|\leq CR^{-1}$ and $|\nabla_{g_0}^2 u|\leq CR^{-2}$ on $B_{3R}$, we have
	\begin{align*}
	\int_{B_{2R}}|\Delta_{g_0}\psi|\,dv_{g_0} = \int_{B_{2R}} \Delta_{g_0}\psi\,dv_{g_0} \leq \int_{B_{3R}}\eta\Delta_{g_0} \psi\,dv_{g_0} & = \int_{B_{3R}}\psi\Delta_{g_0}\eta\,dv_{g_0} \nonumber \\
	& \leq CR^{-2}\|\psi\|_{L^1(B_{3R})}.
	\end{align*}
	Taking $\psi = u + C_1\phi$ and recalling the first sentence in Remark \ref{63}, we therefore see that $\nabla_{g_0}^2 u$ satisfies the desired $L^1$ estimate on $B_{2R}$.
\end{proof}

\begin{rmk}
	We note that the local second derivative estimate of Theorem \ref{61} has previously been obtained for certain choices of $A$ and $B$ using Pogorelov-type arguments. To provide one notable example, Theorem \ref{61} yields local second derivative estimates for the trace-modified Hessian equation 
	\begin{align}\label{86}
	(\sigma_k^{1/k})^\tau\big(\lambda(\nabla_{g_0}^2 u)\big) = f(x,u)>0, \quad \lambda(\nabla_{g_0}^2 u)\in(\Gamma_k^+)^\tau, \quad \tau\in(0,1),
	\end{align}
	which were also obtained using pointwise methods by Guan in \cite[Theorem 3.1]{Guan08}, therein taking $t=s=0$ and $A=0$. In contrast, such estimates for $k$-convex solutions to the $k$-Hessian equation ($t=1$) remain unknown for $k=2$ when $n\geq 4$, and fail for $k\geq 3$ \cite{Pog78, Urb90}. Local second derivative estimates in the Euclidean setting are known for $k=2$ when $n=2$ \cite{Hei59} and $n=3$ \cite{WY09} (see also \cite{Qiu17}). For related work under additional convexity assumptions, see for example \cite{GQ19, MSY19, SY20}.
\end{rmk}

\section{Proof of Theorem \ref{B}}\label{77}

In this section we prove Theorem \ref{B} using an integrability improvement argument followed by Moser iteration. Since our desired estimates are local, it suffices to consider the case that $\Omega\subset M^n$ is a geodesic ball of radius $2R$ contained inside a single coordinate chart, and $\Omega'$ is the concentric geodesic ball with radius $R$. Our equation of interest is therefore
\begin{align}\label{28}
\sigma_2^{1/2}\big(\lambda\big(W[u]\big)\big) = B[u], \quad \lambda\big(W[u]\big)\in\Gamma_2^+ \quad\text{on }B_{2R},
\end{align}
where $W[u] = \nabla^2 u - A[u]$, $A[u](x) = A(x,u(x),du(x))$ defines a symmetric $(0,2)$-tensor at each $x\in B_{2R}$ and $B[u](x) = B(x,u(x),du(x))$ is real and positive. We make no other assumptions on either $A$ or $B$, and we work in arbitrary dimension, although we reiterate that the bound $p>3n/2$ in Theorem \ref{B} only improves on Theorem \ref{A} when $n=2$ or 3.

We will often use the following basic estimates without explicit reference. First, we recall that the assumption $\lambda(W[u])\in\Gamma_2^+$ implies the existence of a constant $\alpha>0$ (depending on $A$ and an upper bound for $ \|u\|_{C^1(B_{2R})}$) for which 
\begin{equation}\label{-15}
0\leq |\nabla^2 u| <\Delta u + \alpha
\end{equation}
and
\begin{equation*}
0<\operatorname{tr}(W[u]) \leq \Delta u + \alpha
\end{equation*}
in $B_{2R}$. In addition, we also assume that $\alpha$ is chosen such that $\Delta u + \alpha \geq 1$. We denote by $F_i^{j}$ the linearisation of $\sigma_2$, which we recall is equal to the first Newton tensor:
\begin{equation}\label{-12}
F_i^{j} = \frac{\partial\sigma_2(\lambda(W[u]))}{\partial(W[u])_{j}^i} = \sigma_1(\lambda(W[u]))\delta_i^{j} - W[u]_i^{j}. 
\end{equation}
It is easy to see from \eqref{-15} and \eqref{-12} that $|F| \leq C(\Delta u+\alpha)$ in $B_{2R}$. 

\begin{rmk}
	We point out that in this section, $F$ denotes the linearisation of $\sigma_2$ rather than the linearisation of $\sigma_2^{1/2}$ (cf.~\eqref{95}). The reason for this change is due to the favourable divergence structure of \eqref{-12}.
\end{rmk}

The plan of the section is as follows. In Section \ref{80} we carry out a series of integral estimates with a view to obtaining the estimate \eqref{81} below. In Section \ref{t50}, we show that under the assumptions of Theorem \ref{B}, the estimate \eqref{81} implies an estimate of reverse H\"older type, which can then be iterated to yield the desired pointwise estimate.  In this section, all computations will be carried out with respect to the background metric $g_0$, and all integrals will be with respect to the volume form of $g_0$.

\subsection{Integral estimates}\label{80}

~\medskip 

\noindent Our starting point in the proof of Theorem \ref{B} is the the following pointwise estimate for solutions $u$ to \eqref{28}, which follows from the same argument as in the start of the proof of Lemma 2.1:
\begin{equation}\label{-11}
2B\Delta B[u]  \leq F^{ij} \nabla_i\nabla_j \Delta u  -\ F^{ij}\Delta A[u]_{ij}+ C(\Delta u + \alpha)^2\quad\mathrm{on~}B_{2R}.
\end{equation}

For $\rho\in(0,\frac{R}{2}]$ we now let $\eta\in C_0^\infty(B_{R+2\rho})$ be a non-negative cutoff function satisfying $0\leq \eta \leq 1$, $\eta=1$ on $B_{R+\rho}$ and $|\nabla^l\eta|\leq C(n)\rho^{-l}$ for $l=1,2$. Multiplying both sides of \eqref{-11} by $\eta(\Delta u + \alpha)^{q-1}$ (where $q>1$ is to be determined) and integrating over $B_{R+2\rho}$, we obtain the integral estimate
\begin{align}\label{-27}
2\int_{B_{R+2\rho}} \eta&(\Delta u +\alpha)^{q-1} B\Delta B[u]
\nonumber \\
& \leq \int_{B_{R+2\rho}} \eta(\Delta u+\alpha)^{q-1} F^{ij} \nabla_i\nabla_j \Delta u  -\int_{B_{R+2\rho}}\eta(\Delta u+\alpha)^{q-1} F^{ij}\Delta A[u]_{ij} \nonumber \\
& \quad + C\int_{B_{R+2\rho}}(\Delta u + \alpha)^{q+1}.
\end{align} 
Note that here and henceforth, $C$ is a constant depending on $n,g_0, A, B, R$ and an upper bound for $\|u\|_{C^1(B_{2R})}$, but $C$ will remain independent of the size of $q$ and $\rho^{-1}$ (this will be important for the iteration argument in Section \ref{t50}). We do not allow $C$ to depend on any norms of second derivatives of $u$, as the dependence on $\|\nabla_{g_0}^2 u\|_{L^p(B_{2R},g_0)}$ in our final estimates will be explicit. $C$ may continue to change from line to line, and we continue to implicitly raise and lower indices using the background metric $g_0$ and its inverse. At various points in our argument, we will implicitly use the fact that $\Delta u + \alpha \geq 1$. \medskip

Our main goal in this section is to prove the following proposition: 

\begin{prop}\label{17}
	Suppose $q>1$, $B\in C^2_{\operatorname{loc}}(B_{2R}\times\mathbb{R}\times T^*B_{2R})$ is positive and $u\in C^4(B_{2R})$ is a solution to \eqref{28}. Then 
	\begin{equation}\label{81}
	\int_{B_{R+\rho}}\big|\nabla(\Delta u+\alpha)^{(q-1)/2}\big|^2\leq \frac{C(q-1)}{\rho^2}\int_{B_{R+2\rho}}(\Delta u+\alpha)^{q+2}.
	\end{equation}
\end{prop}

We will prove Proposition \ref{17} through a series of lemmas. We first prove: 

\begin{lem}\label{t8}
	Under the same hypotheses as Proposition \ref{17},
	\begin{align}\label{-41}
	&(q-1)\int_{B_{R+2\rho}}\eta(\Delta u+\alpha)^{q-2}F^{ij}\nabla_i\Delta u \nabla_j\Delta u - \frac{1}{q}\int_{B_{R+2\rho}}(\Delta u + \alpha)^q \eta \nabla_j \nabla_i F^{ij} \nonumber \\
	&   + \int_{B_{R+2\rho}}\eta(\Delta u+\alpha)^{q-1} F^{ij}\Delta A[u]_{ij}  \leq C\rho^{-2}\int_{B_{R+2\rho}}(\Delta u + \alpha)^{q+1}. 
	\end{align}
\end{lem}

\begin{proof}
	Integrating by parts in the first integral on the RHS of \eqref{-27}, we see
	\begin{align}\label{-35}
	\int_{B_{R+2\rho}} \eta& (\Delta u+\alpha)^{q-1}F^{ij} \nabla_i\nabla_j \Delta u  \nonumber \\
	& = -\int_{B_{R+2\rho}}\eta(\Delta u+\alpha)^{q-1}\nabla_i F^{ij}\nabla_j\Delta u  -\int_{B_{R+2\rho}}(\Delta u+\alpha)^{q-1} F^{ij}\nabla_i\eta \nabla_j\Delta u \nonumber \\
	& \quad  - (q-1)\int_{B_{R+2\rho}}\eta(\Delta u+\alpha)^{q-2}F^{ij}\nabla_i\Delta u \nabla_j\Delta u.
	\end{align}
	Integrating by parts again, the first integral on the RHS of \eqref{-35} is 
	\begin{align}\label{44}
	-\int_{B_{R+2\rho}}&\eta(\Delta u+\alpha)^{q-1}\nabla_i F^{ij}\nabla_j\Delta u \nonumber \\
	& = -\frac{1}{q}\int_{B_{R+2\rho}}\eta\nabla_i F^{ij} \nabla_j(\Delta u+\alpha)^q \nonumber \\
	& = \frac{1}{q}\int_{B_{R+2\rho}}(\Delta u +\alpha)^q \nabla_i F^{ij}\nabla_j\eta + \frac{1}{q}\int_{B_{R+2\rho}}(\Delta u + \alpha)^q \eta \nabla_j \nabla_i F^{ij}.
	\end{align}
	Similarly, the second integral on the RHS of \eqref{-35} can be further computed as follows:
	\begin{align}\label{-40}
	-\int_{B_{R+2\rho}}&(\Delta u+\alpha)^{q-1} F^{ij}\nabla_i\eta \nabla_j\Delta u \nonumber \\
	&   = -\frac{1}{q}\int_{B_{R+2\rho}}F^{ij}\nabla_i\eta \nabla_j(\Delta u+\alpha)^{q} \nonumber \\
	& = \frac{1}{q}\int_{B_{R+2\rho}}(\Delta u+\alpha)^{q}\nabla_i \eta\nabla_j F^{ij} + \frac{1}{q}\int_{B_{R+2\rho}}(\Delta u+\alpha)^{q}F^{ij}\nabla_j\nabla_i\eta \nonumber \\
	& \geq \frac{1}{q}\int_{B_{R+2\rho}}(\Delta u+\alpha)^{q}\nabla_i \eta\nabla_j F^{ij} - C\rho^{-2}\int_{B_{R+2\rho}}(\Delta u + \alpha)^{q+1}.
	\end{align}
	Substituting \eqref{44} and \eqref{-40} into \eqref{-35}, and then \eqref{-35} back into \eqref{-27} and rearranging, we obtain 
	\begin{align}\label{-42}
	&(q-1)\int_{B_{R+2\rho}}\eta(\Delta u+\alpha)^{q-2}F^{ij}\nabla_i\Delta u \nabla_j\Delta u  - \frac{1}{q}\int_{B_{R+2\rho}}(\Delta u + \alpha)^q \eta \nabla_j \nabla_i F^{ij}   \nonumber \\
	&\quad + \int_{B_{R+2\rho}}\eta(\Delta u+\alpha)^{q-1} F^{ij}\Delta A[u]_{ij} \leq  \frac{2}{q}\int_{B_{R+2\rho}}(\Delta u+\alpha)^{q}\nabla_i F^{ij}\nabla_j \eta \nonumber \\
	& \quad  - 2\int_{B_{R+2\rho}} \eta(\Delta u +\alpha)^{q-1} B\Delta B[u]+ C\rho^{-2}\int_{B_{R+2\rho}}(\Delta u+\alpha)^{q+1}.
	\end{align}
	Note that the LHS of \eqref{-42} is precisely the LHS of \eqref{-41}. Therefore, to obtain \eqref{-41} it remains to estimate the first two integrals on the RHS of \eqref{-42} from above by $C\rho^{-2}\int_{B_{R+2\rho}}(\Delta u + \alpha)^{q+1}$. For the first of these, we calculate the divergence of $F$:
	\begin{align}\label{-47}
	\nabla_i F^{ij} & = \nabla_i \big(\operatorname{tr}(W[u])\delta^{ij} - (W[u])^{ij} \big)  \nonumber \\
	& = \nabla^j \big(\Delta u - \operatorname{tr}(A[u])\big)- \nabla_i \big(\nabla^i\nabla^j u - A[u]^{ij}\big) \nonumber \\
	& \leq C(\Delta u + \alpha),
	\end{align}
	where to reach the last line we have commuted derivatives to assert $|\nabla^j\Delta u - \nabla_i\nabla^i\nabla^j u| \leq C|\nabla u| \leq C$. The desired estimate for the first integral on the RHS of \eqref{-42} then follows immediately.
	
	For the second integral on the RHS of \eqref{-42}, by the calculation in \eqref{20} with $B[u]$ in place of $A[u]_{ij}$, we see
	\begin{align*}
	- 2\int_{B_{R+2\rho}} &\eta(\Delta u  +\alpha)^{q-1} B\Delta B[u]  \nonumber \\
	&  \leq -2\int_{B_{R+2\rho}}\eta(\Delta u + \alpha)^{q-1} B \frac{\partial B}{\partial \xi_l}[u]\nabla_l\Delta u + C\int_{B_{R+2\rho}}(\Delta u + \alpha)^{q+1} \nonumber \\
	& = -\frac{2}{q}\int_{B_{R+2\rho}}\eta B\frac{\partial B}{\partial \xi_l}[u]\nabla_l(\Delta u + \alpha)^{q} + C\int_{B_{R+2\rho}}(\Delta u + \alpha)^{q+1} \nonumber \\
	& \leq C\rho^{-2}\int_{B_{R+2\rho}}(\Delta u + \alpha)^{q+1},
	\end{align*}
	where to reach the last line we have integrated by parts and used properties of $\eta$. 
\end{proof}

With Proposition \ref{17} in mind, we now deal with the integrals on the LHS of \eqref{-41}, starting with a pointwise estimate for the first integrand:

\begin{lem}\label{t9}
	Under the same hypotheses as Proposition \ref{17},
	\begin{equation}\label{n14}
	(\Delta u+\alpha)^{q-2}F^{ij}\nabla_i \Delta u\nabla_j\Delta u \geq \frac{CB[u]^2}{(q-1)^2}\Big|\nabla(\Delta u+\alpha)^{(q-1)/2}\Big|^2.
	\end{equation}
\end{lem}

\begin{proof}
	The proof is similar to that given in \cite{Urb00}; we give the argument here for completeness. Let $F^{ij}_{(l)}(X)$ denote the matrix with entries $\partial\sigma_l(\lambda(X))/\partial X_{ij}$. Then by concavity of $\sigma_k(\lambda(X))/\sigma_{k-1}(\lambda(X))$ on the set of symmetric matrices with $\lambda(X)\in\Gamma_k^+$, on this set we have
	\begin{align}\label{82}
	\frac{F_{(k)}^{ij}(X)}{\sigma_k(\lambda(X))} \geq \frac{F_{(k-1)}^{ij}(X)}{\sigma_{k-1}(\lambda(X))} \geq \dots \geq \frac{F_{(1)}^{ij}(X)}{\sigma_1(\lambda(X))} = \frac{\delta^{ij}}{\sigma_1(\lambda(X))}.
	\end{align}
	Taking $X = W[u]$ and $k=2$ in \eqref{82}, and applying the equation \eqref{28}, it follows that
	\begin{equation*}
	F^{ij} \geq \frac{B[u]^2\delta^{ij}}{\operatorname{tr}(W[u])} \geq \frac{B[u]^2\delta^{ij}}{\Delta u + \alpha}, 
	\end{equation*}
	and hence
	\begin{align*}
	(\Delta u+\alpha)^{q-3}F^{ij}\nabla_i \Delta u\nabla_j\Delta u & = \frac{4}{(q-1)^2} F^{ij}\nabla_i(\Delta u + \alpha)^{(q-1)/2}\nabla_j(\Delta u + \alpha)^{(q-1)/2} \nonumber \\
	& \geq \frac{4B[u]^2}{(q-1)^2}\frac{1}{\Delta u + \alpha}\Big|\nabla(\Delta u+\alpha)^{(q-1)/2}\Big|^2. 
	\end{align*}
	Multiplying through by $\Delta u + \alpha$, we arrive at \eqref{n14}. 
\end{proof}

We next prove the following estimate for the third integral on the RHS of \eqref{-41}: 

\begin{lem}
	Under the same hypotheses as Proposition \ref{17},
	\begin{align}\label{t10}
	\int_{B_{R+2\rho}}\eta(\Delta u + \alpha)^{q-1}F^{ij}\Delta A[u]_{ij}& \geq - \int_{B_{R+2\rho}}\eta(\Delta u + \alpha)^{q-1}\nabla^i\nabla^j u\frac{\partial A_{ij}}{\partial\xi_a}[u]\nabla_a\Delta u \nonumber \\
	& \quad  - C\rho^{-2}\int_{B_{R+2\rho}}(\Delta u + \alpha)^{q+2}. 
	\end{align}
\end{lem}

\begin{proof} 
By \eqref{20},
\begin{align}\label{51}
F^{ij}\Delta A[u]_{ij}  &  \geq - C(\Delta u + \alpha)^3 + F^{ij}\frac{\partial A_{ij}}{\partial \xi_l}[u]\Delta\nabla_l u \nonumber \\
& \geq - C(\Delta u + \alpha)^3 + F^{ij}\frac{\partial A_{ij}}{\partial \xi_l}[u]\nabla_l\Delta u,
\end{align}
where to reach the second line we have commuted derivatives and absorbed resulting curvature terms into the $-C(\Delta u + \alpha)^3$ term. Substituting \eqref{51} into the LHS of \eqref{t10} and then using the identity $F^{ij} = \operatorname{tr}(W[u])\delta^{ij} - \nabla^i\nabla^j u  + A^{ij}$, we have
	\begin{align}\label{52}
\int_{B_{R+2\rho}}&\eta(\Delta u + \alpha)^{q-1}F^{ij}\Delta A[u]_{ij} \nonumber \\
& \geq \int_{B_{R+2\rho}}\eta(\Delta u + \alpha)^{q-1}F^{ij}\frac{\partial A_{ij}}{\partial\xi_a}[u]\nabla_a\Delta u   - C\rho^{-2}\int_{B_{R+2\rho}}(\Delta u + \alpha)^{q+2} \nonumber \\
& = \int_{B_{R+2\rho}}\eta(\Delta u + \alpha)^{q-1}\operatorname{tr}(W[u])\frac{\partial\operatorname{tr}(A)}{\partial\xi_a}[u]\nabla_a\Delta u \nonumber \\
& \quad - \int_{B_{R+2\rho}}\eta(\Delta u + \alpha)^{q-1}\nabla^i\nabla^j u \frac{\partial A_{ij}}{\partial \xi_a}[u]\nabla_a\Delta u \nonumber \\
& \quad + \int_{B_{R+2\rho}}\eta(\Delta u + \alpha)^{q-1}A^{ij}\frac{\partial A_{ij}}{\partial \xi_a}[u]\nabla_a\Delta u - C\rho^{-2}\int_{B_{R+2\rho}}(\Delta u + \alpha)^{q+2}. 
\end{align}
Now, since $\operatorname{tr}(W[u]) = (\Delta u + \alpha) - \operatorname{tr}(A) - \alpha$, the term on the third line of \eqref{52} can be written as 
\begin{align}\label{73}
&\int_{B_{R+2\rho}}\eta(\Delta u + \alpha)^{q-1}\operatorname{tr}(W[u])\frac{\partial\operatorname{tr}(A)}{\partial\xi_a}[u]\nabla_a\Delta u \nonumber \\
& = \int_{B_{R+2\rho}}\eta(\Delta u + \alpha)^{q}\frac{\partial\operatorname{tr}(A)}{\partial\xi_a}[u]\nabla_a\Delta u - \int_{B_{R+2\rho}}\eta(\Delta u + \alpha)^{q-1}\operatorname{tr}(A)\frac{\partial\operatorname{tr}(A)}{\partial\xi_a}[u]\nabla_a\Delta u \nonumber \\
& \quad -\alpha\int_{B_{R+2\rho}}\eta(\Delta u + \alpha)^{q-1}\frac{\partial\operatorname{tr}(A)}{\partial\xi_a}[u]\nabla_a\Delta u 
\nonumber \\
& = \frac{1}{q+1}\int_{B_{R+2\rho}}\eta\frac{\partial\operatorname{tr}(A)}{\partial\xi_a}[u]\nabla_a(\Delta u+\alpha)^{q+1} - \frac{1}{q} \int_{B_{R+2\rho}}\eta \operatorname{tr}(A)\frac{\partial\operatorname{tr}(A)}{\partial\xi_a}[u]\nabla_a(\Delta u+\alpha)^q \nonumber \\
& \quad -\frac{\alpha}{q}\int_{B_{R+2\rho}}\eta \frac{\partial\operatorname{tr}(A)}{\partial\xi_a}[u]\nabla_a(\Delta u +\alpha)^q.
\end{align}
After integrating by parts in each of the last three integrals in \eqref{73}, we therefore see that 
\begin{align}\label{74}
\int_{B_{R+2\rho}}\eta(\Delta u + \alpha)^{q-1}\operatorname{tr}(W[u])\frac{\partial\operatorname{tr}(A)}{\partial\xi_a}[u]\nabla_a\Delta u \geq -C\rho^{-2}\int_{B_{R+2\rho}}(\Delta u + \alpha)^{q+1}. 
\end{align}
Likewise, writing the penultimate term in \eqref{52} as
\begin{equation*}
\frac{1}{q}\int_{B_{R+2\rho}}\eta A^{ij}\frac{\partial A_{ij}}{\partial\xi_a}[u]\nabla_a(\Delta u + \alpha)^q
\end{equation*}
and integrating by parts, we observe that this term is also bounded from below by \linebreak  $-C\rho^{-2}\int_{B_{R+2\rho}}(\Delta u + \alpha)^{q+1}$. Substituting this estimate and \eqref{74} back into \eqref{52}, we arrive at \eqref{t10}. 
\end{proof} 

Before giving the proof of Proposition \ref{17}, it remains to estimate the second integral on the LHS of \eqref{-41}:

\begin{lem}
	Under the same hypotheses as Proposition \ref{17},
	\begin{align}\label{13}
	- \frac{1}{q}\int_{B_{R+2\rho}}(\Delta u + \alpha)^q \eta \nabla_j \nabla_i F^{ij} &  \geq  - \frac{1}{q}\int_{B_{R+2\rho}}\eta(\Delta u + \alpha)^q\frac{\partial A_{ij}}{\partial \xi_l}[u]\nabla^i\nabla^j \nabla_l u \nonumber \\
	& \quad - C\rho^{-2}\int_{B_{R+2\rho}}(\Delta u + \alpha)^{q+2}.
	\end{align}
\end{lem}
\begin{proof}
	By \eqref{-47}, we see
	\begin{align*}
	\nabla_j\nabla_i F^{ij} & = \nabla_j\big(\nabla^j\Delta u - \nabla_i\nabla^i\nabla^j u - \nabla^j\operatorname{tr}(A[u]) + \nabla_i A[u]^{ij} \big) \nonumber \\
	& = \Delta^2 u - \nabla_j\nabla_i\nabla^i\nabla^j u - \Delta \operatorname{tr}(A[u]) + \nabla_j\nabla_i A[u]^{ij} \nonumber \\
	& \leq C(\Delta u + \alpha) - \Delta \operatorname{tr}(A[u]) + \nabla_j\nabla_i A[u]^{ij}
	\end{align*}
	where we have commuted derivatives to assert $|\Delta^2 u - \nabla_j\nabla_i\nabla^i\nabla^j u|\leq C|\nabla^2 u| \leq C(\Delta u +\alpha)$. It follows that 
	\begin{align}\label{64}
	- \frac{1}{q}\int_{B_{R+2\rho}}(\Delta u + \alpha)^q \eta \nabla_j \nabla_i F^{ij} & \geq  \frac{1}{q}\int_{B_{R+2\rho}}\eta(\Delta u + \alpha)^q\Delta\operatorname{tr}(A[u]) \nonumber \\
	& \quad - \frac{1}{q}\int_{B_{R+2\rho}}\eta(\Delta u + \alpha)^q\nabla_j\nabla_i A[u]^{ij} \nonumber \\
	& \quad -C\rho^{-2}\int_{B_{R+2\rho}}(\Delta u + \alpha)^{q+1}. 
	\end{align}
	Now, similarly to \eqref{51}, we have 
	\begin{align}
	\Delta \operatorname{tr}(A[u]) \geq - C(\Delta u+\alpha)^2 + \frac{\partial \operatorname{tr}(A)}{\partial\xi_l}[u]\nabla_l\Delta u \nonumber 
	\end{align}
	and
	\begin{align}
	-\nabla_j\nabla_iA[u]^{ij} \geq -C(\Delta u + \alpha)^2 - \frac{\partial A_{ij}}{\partial \xi_l}[u]\nabla^i\nabla^j \nabla_l u. \nonumber
	\end{align}
	Substituting these two inequalities into \eqref{64} yields
	\begin{align}\label{65}
	- \frac{1}{q}\int_{B_{R+2\rho}}(\Delta u + \alpha)^q \eta \nabla_j \nabla_i F^{ij} & \geq \frac{1}{q}\int_{B_{R+2\rho}}\eta(\Delta u + \alpha)^q\frac{\partial \operatorname{tr}(A)}{\partial\xi_l}[u]\nabla_l\Delta u \nonumber \\
	& \quad -  \frac{1}{q}\int_{B_{R+2\rho}}\eta(\Delta u + \alpha)^q\frac{\partial A_{ij}}{\partial \xi_l}[u]\nabla^i\nabla^j \nabla_l u \nonumber \\
	& \quad - C\rho^{-2}\int_{B_{R+2\rho}}(\Delta u + \alpha)^{q+2}.
	\end{align}
	The desired inequality \eqref{13} then follows after writing the first term on the RHS of \eqref{65} as
	\begin{equation*}
	\frac{1}{q(q+1)}\int_{B_{R+2\rho}}\eta \frac{\partial\operatorname{tr}(A)}{\partial\xi_l}[u]\nabla_l(\Delta u + \alpha)^{q+1}
	\end{equation*}
	and integrating by parts to observe that this term is bounded from below by \linebreak  $-C\rho^{-2}\int_{B_{R+2\rho}}(\Delta u + \alpha)^{q+2}$. 
	\end{proof}

We now give the proof of Proposition \ref{17}:

\begin{proof}[Proof of Proposition \ref{17}]
	Substituting the estimates \eqref{n14}, \eqref{t10} and \eqref{13} back into \eqref{-41}, we obtain 
	\begin{align}\label{18}
	\frac{C}{q-1}& \int_{B_{R+2\rho}}\eta B[u]^2\Big|\nabla(\Delta u+\alpha)^{(q-1)/2}\Big|^2 -  \int_{B_{R+2\rho}}\eta(\Delta u + \alpha)^{q-1}\nabla^i\nabla^j u\frac{\partial A_{ij}}{\partial\xi_a}[u]\nabla_a\Delta u \nonumber \\
	& - \frac{1}{q}\int_{B_{R+2\rho}}\eta(\Delta u + \alpha)^q\frac{\partial A_{ij}}{\partial \xi_a}[u]\nabla^i\nabla^j \nabla_a u \leq C\rho^{-2}\int_{B_{R+2\rho}}(\Delta u + \alpha)^{q+2}. 
	\end{align}
	But 
	\begin{align}\label{75}
 - & \int_{B_{R+2\rho}}\eta(\Delta u + \alpha)^{q-1}\nabla^i\nabla^j u\frac{\partial A_{ij}}{\partial\xi_a}[u]\nabla_a\Delta u  \nonumber \\
 & = -\frac{1}{q}\int_{B_{R+2\rho}}\eta\nabla^i\nabla^j u\frac{\partial A_{ij}}{\partial\xi_a}[u]\nabla_a(\Delta u + \alpha)^q \nonumber \\
 & = \frac{1}{q}\int_{B_{R+2\rho}}\eta(\Delta u + \alpha)^q\frac{\partial A_{ij}}{\partial \xi_a}[u]\nabla^i\nabla^j \nabla_a u + \frac{1}{q}\int_{B_{R+2\rho}}(\Delta u + \alpha)^q\nabla^i\nabla^j u \nabla_a\bigg(\eta \frac{\partial A_{ij}}{\partial\xi_a}[u]\bigg) \nonumber \\
 & \geq \frac{1}{q}\int_{B_{R+2\rho}}\eta(\Delta u + \alpha)^q\frac{\partial A_{ij}}{\partial \xi_a}[u]\nabla^i\nabla^j \nabla_a u - C\rho^{-2}\int_{B_{R+2\rho}}(\Delta u + \alpha)^{q+2}.
	\end{align}
	After substituting \eqref{75} back into \eqref{18} and observing cancellation with the third term in \eqref{18}, we obtain the estimate 
	\begin{align}\label{83}
	\frac{C}{q-1}& \int_{B_{R+2\rho}}\eta B[u]^2\Big|\nabla(\Delta u+\alpha)^{(q-1)/2}\Big|^2  \leq C\rho^{-2}\int_{B_{R+2\rho}}(\Delta u + \alpha)^{q+2}. 
	\end{align}
 The desired estimate \eqref{17} then follows after absorbing $\inf_{B_{R+2\rho}}B[u]^2$ into the constant on the LHS of \eqref{83} and using properties of $\eta$.  
\end{proof}

\subsection{Integrability improvement and Moser iteration}\label{t50}

~\medskip 

\noindent In this section we complete the proof of Theorem \ref{B}. 

\begin{proof}[Proof of Theorem \ref{B}]
First observe that by the Sobolev inequality applied to the function $(\Delta u + \alpha)^{(q-1)/2}$, we have
	\begin{align}\label{19}
	\bigg(\int_{B_{R+\rho}}(\Delta u+\alpha)^{\beta(q-1)}\bigg)^{1/\beta} & \leq C\int_{B_{R+\rho}}\big|\nabla(\Delta u+\alpha)^{(q-1)/2}\big|^2 + C\int_{B_{R+\rho}}(\Delta u+\alpha)^{q-1},
	\end{align}
	where $\beta = \frac{n}{n-2}$ if $n\geq 3$ and $\beta>1$ is any finite number if $n=2$. Note that the constants $C$ in \eqref{19} depend on the choice of $\beta$ when $n=2$, but $\beta$ will ultimately be fixed. Substituting the estimate \eqref{17} from Proposition \ref{17} into the RHS of \eqref{19}, it follows that
	\begin{equation}\label{t51}
	\bigg(\int_{B_{R+\rho}}(\Delta u+\alpha)^{s\beta}\bigg)^{1/\beta} \leq \frac{Cs}{\rho^2}\int_{B_{R+2\rho}}(\Delta u+\alpha)^{s+3}
	\end{equation}
	where $s=q-1$. 
	
Let us first address the case $n\geq 3$. It is clear that \eqref{t51} yields an improvement in integrability whenever $s+3>\frac{3n}{2}$. Under this assumption, it remains to iterate this improvement in integrability to get the pointwise estimate. Let $p>\frac{3n}{2}$ be as in the statement of Theorem \ref{B}, and define the sequence $s_j$ inductively by 
	\begin{equation*}
	s_0 = p - 3 \quad\text{and}\quad s_j = \beta s_{j-1} - 3 \text{ for }j\geq 1. 
	\end{equation*}
	Then $s_j = \beta^j s_0 - 3(\beta^{j-1}+\dots+\beta+1) = \beta^j s_0 - 3(1-\beta^j)/(1-\beta)$ and hence
	\begin{align}\label{79}
	\frac{s_j}{\beta^j} = s_0 - \frac{3(1-\beta^{-j})}{\beta-1}\rightarrow s_0 - \frac{3}{\beta-1}>0\text{ as }j\rightarrow\infty,
	\end{align}
	where we have used the definition of $\beta$ and the fact that $s_0 = p - 3 > \frac{3n}{2}-3$ to assert positivity in \eqref{79}. It follows from \eqref{79} that $s_j\rightarrow \infty$ as $j\rightarrow \infty$. We now apply \eqref{t51} iteratively with $s=s_j$ and $\rho=2^{-j-1}R$ to obtain 
	\begin{align}\label{84}
	\bigg(\int_{B_{(1+2^{-j-1})R}} (\Delta u +\alpha)^{\beta s_j}\bigg)^{\beta^{-j-1}} & \leq \bigg(4^j C s_j \int_{B_{(1+2^{-j})R}}(\Delta u + \alpha)^{\beta s_{j-1}}\bigg)^{\beta^{-j}} \nonumber \\
	& \leq \prod_{i=0}^j \big((4\beta)^iC\big)^{\beta^{-i}}\int_{B_{2R}}(\Delta u + \alpha)^p \nonumber \\
	& \leq (4\beta)^{\sum_{i=0}^\infty i\beta^{-i}}C^{\sum_{i=0}^\infty \beta^{-i}}\int_{B_{2R}}(\Delta u + \alpha)^p
	\end{align}
	we have used \eqref{79} in obtaining the second inequality. Letting $j\rightarrow\infty$ in \eqref{84} and again using \eqref{79}, we arrive at the estimate
	\begin{align*}
	\|\Delta u + \alpha\|_{L^\infty(B_R)} \leq C\bigg(\int_{B_{2R}}(\Delta u + \alpha)^p\bigg)^{(s_0 - \frac{3}{\beta-1})^{-1}}.
	\end{align*}
	As explained in Remark \ref{63}, the desired Hessian bound then follows. 
	
	Finally, if $n=2$, then we see that \eqref{t51} yields an improvement in integrability for any given $s>0$ if $\beta$ is chosen such that $\beta>\frac{s+3}{s}$. After such a choice of $\beta$ is made, the Moser iteration procedure may be followed as above. 
\end{proof}

\begin{rmk}\label{72}
We note that the method for proving Theorem \ref{B} can be extended to deal with the case $(f,\Gamma)= ((\sigma_2^{1/2})^\tau,(\Gamma_2^+)^\tau)$ for $\tau<1$, but still under the assumption $p>3n/2$. Since this yields no improvement on Theorem \ref{A} in any dimension, we omit the proof.
\end{rmk}

\bibliography{references}{}
\bibliographystyle{siam}

\end{document}